\documentclass[11pt,reqno]{amsart}
\usepackage{amsmath,amssymb}
\usepackage[leqno]{amsmath}
\usepackage{agor}
\usepackage{color}
\usepackage{bbm}
\usepackage{lscape}
\usepackage{tablefootnote}
\usepackage[square,numbers,sort]{natbib}
\usepackage{todonotes}
\usepackage{multirow}
\usepackage{mathtools}
\usepackage{algcompatible}
\usepackage{geometry}
\usepackage{hyperref}
\newcommand{\leqnomode}{\tagsleft@true\let\veqno\@@leqno}
\newcommand{\reqnomode}{\tagsleft@false\let\veqno\@@eqno}

\usepackage{dsfont}
\title[An Outer Approximation Method for MICQPs]{An Outer Approximation Method for Solving Mixed-Integer Convex Quadratic Programs with Indicators}
\author{Linchuan Wei,  Simge K\"u\c{c}\"ukyavuz}

\thanks{ \noindent \hskip -5mm
L. Wei: {Department of Industrial Engineering and Management Sciences, Northwestern University, \texttt{linchuanwei2022@u.northwestern.edu}.}\\
S. K\"u\c{c}\"ukyavuz: {Department of Industrial Engineering and Management Sciences, Northwestern University, \texttt{simge@northwestern.edu}.}\\
}

\begin{document}
\maketitle

\begin{abstract}

Mixed-integer convex quadratic programs with indicator variables (MIQP) encompass a wide range of applications, from statistical learning to energy, finance, and logistics. The outer approximation (OA) algorithm has been proven efficient in solving MIQP, and the key to the success of an OA algorithm is the strength of the cutting planes employed. In this paper, we propose a new technique for deriving cutting planes for MIQP from various convex relaxations, and, as a result, we develop new OA algorithms for solving MIQP at scale. The contributions of our work are two-fold: (1) we bridge the work on the convexification of MIQP and the algorithm design to solve large-scale problems, and (2) we demonstrate through a computational study on the sparse portfolio selection problem that our algorithms give rise to significant speedups compared with the state-of-the-art methods in the literature. 
\end{abstract}

\section{Introduction}
Mixed-integer convex quadratic programs with indicators (MIQP) constitute a broad class of optimization problems formulated as:
\begin{subequations}\label{eqn:MIQP}
\begin{align}
    \min_{x, y} \quad &   y^{\top} Q y + g^{\top} y + h^{\top} x\\
    \text{s.t.} \quad &  A y \leq b \label{MIQP:const1}\\
     \text{(MIQP)} \qquad \qquad         &  C y \leq D x \label{MIQP:const2}\\
                &  y_i (1 - x_i) = 0, \quad \forall i = 1, \dots, n \label{MIQP:const3}\\
                &  x \in X \subseteq \{0,1\}^{n} \label{MIQP:const4},
\end{align}
\end{subequations}
where $Q \in \R^{n\times n}$ is a positive definite Hessian matrix, and $X$ denotes an arbitrary set over the binary variable vector $x$. The complementarity constraint \eqref{MIQP:const3} encodes the relation that the indicator variable $x_i$ activates/deactivates the associated continuous variable $y_i, i=1,\dots,n$. We also consider arbitrary linear constraints over $y$ and linear linking constraints over $x$ and $y$. Examples of constraints \eqref{MIQP:const1}--\eqref{MIQP:const2} include the nonnegativity constraints ($y_i \geq 0$) and the semi-continuity constraints ($y_i = 0$ if $x_i = 0$, and $y_i \in [l_i, u_i]$ if $x_i = 1$, which can be modeled as $l_i x_i \leq y_i \leq u_i x_i$). Examples of $X$ include the cardinality constraint, i.e., $X = X_k:=\{x \in \{0,1\}^{n} : \sum_{i = 1}^{n} x_i \leq k\},$ for a given $k<n$. MIQP finds a wide range of applications in statistical learning with sparsity \cite{Miller2002subset, Bertsimas2016, bertsimas2020sparse, MKS21,KSMW23}, portfolio optimization \cite{BD:miqp, bertsimas2022scalable}, energy production \cite{Frangioni2006}, and logistics \cite{Gunluk2010, fischetti2017redesigning}. On the other hand, MIQP is $\mathcal{NP}$-hard even when $x$ and $y$ are unconstrained \cite{CGWY:2014complexity}.

\subsection{Background and Literature review}

By introducing an auxiliary variable $\eta$, we can write problem \eqref{eqn:MIQP} equivalently as follows
\begin{align*}
    \min_{} \quad &  \eta \\
    \text{s.t.} \quad & \eta \geq y^{\top} Q y + g^{\top} y + h^{\top} x\\
                &  A y \leq b \\
                &  C y \leq Dx \\\
                &  y_i (1 - x_i) = 0, \quad \forall i =1, \dots,n \\
                &  x \in X \subseteq \{0,1\}^{n} ,
\end{align*}
and we define an epigraph accordingly
\begin{align*}
\mathcal{Z} = \{(x, y, \eta)  \in X \times \R^{n+1} \; | \; & \eta \geq y^{\top} Q y + g^{\top} y + h^{\top} x,  \\
& Ay \leq b, C y \leq D x, y_i (1 - x_i) = 0, \; \forall i = 1, \dots, n\}.
\end{align*}
Note that minimizing $\eta$ over $(x, y, \eta) \in \mathcal Z$ for some $x$ and $y$ is equivalent to minimizing over the closure of the convex hull of $\mathcal Z$ denoted as $\clconv(\mathcal Z)$. In recent years, there has been a growing interest in studying the convexification of $\mathcal Z$ via decomposition. For example, the Hessian matrix $Q$ can be decomposed into $\sum_{i} \Gamma_i + R$, where each $\Gamma_i$ exhibits some `simple' structure. Then, a convex relaxation of $\mathcal Z$ can be derived by studying the mixed-integer quadratic epigraphs where each $\Gamma_i$ is the Hessian of the associated quadratic term. Specifically, the perspective reformulation \cite{Ceria1999,Frangioni2006} is obtained when we decompose the quadratic function into a sum of one-dimensional quadratic terms and a remainder ($Q = \text{diag}(\delta) + R$ such that $\delta \geq 0$ and $R \succcurlyeq0$) and replace $\delta_i y_i^2$ with $\delta_i \frac{y_i^2}{x_i}$, and $\frac{y_i^2}{x_i}$ is known as the perspective function of $y_i^2$ \cite{Hiriart2013}.  Other convexification work considers decomposing $Q$ into rank-one matrices \cite{atamturk2019rank, wei2020convexification,wei2022ideal, gomez2018submodularity},  two-by-two matrices \cite{frangioni2018decompositions, HGA:2x2, gomez2018strong, Jeon2017}, and tridiagonal matrices \cite{liu2021graph}. In principle, $\clconv(\mathcal Z)$ can be obtained directly using the disjunctive programming approach \cite{Ceria1999, frangioni2018decompositions}; however, a straightforward disjunctive formulation requires a copy of variables for each configuration of indicators in $X$, which results in an exponential number of auxiliary variables. \citet{wei2023convex} give a more compact description of $\clconv(\mathcal Z)$ in the absence of the linear constraints \eqref{MIQP:const1}--\eqref{MIQP:const2}.  Another related strand of research concerns the convexification of the lifted set $S = \{(x, y, y^{\top}y) \in \{0,1\}^{n} \times \R^{n}\times\R^{n \times n}:y_i(1-x_i) = 0, \forall i =1, \dots,n\}$, and $S$ stems from the more general mixed-integer quadratic programs with indicators.  \citet{anstreicher2021quadratic} deliver a convex hull description for $S$ with bound constraints $0 \leq y \leq x$ when $n = 2$. \citet{de2022explicit} give the convex hull of $S$ with nonnegativity constraints for $n = 2$. Although there are many strong theoretical results on the convexification of MIQP, a stronger reformulation usually involves a larger set of auxiliary variables, more constraints, or more nonlinearity in the objective or constraints, which may hinder the branch-and-bound process.

Regarding the exact solution methods for MIQP, substantial research emerged in recent years which can be categorized into (i) branch-and-bound methods that solve a sequence of subproblems and (ii) outer-approximation (OA) algorithms that recursively solve polyhedral approximations of $\mathcal Z$. 

\vspace{1em}

\noindent \textbf{Branch-and-bound}: In the literature, various branch-and-bound methods using different reformulations have been proposed for solving MIQP (or its special cases) to optimality. \citet{BD:miqp} first proposed solving MIQP via branch-and-bound, and, instead of using indicators, their algorithm directly branches on the continuous variables by enforcing bound constraints $y_i \leq 0$ or $y_i \geq \alpha_i$. \citet{Bertsimas2016} modeled the complementarity constraints \eqref{MIQP:const3} using Specially Ordered Sets of Type 1 (SOS-1) \cite{bertsimas2005optimization} and solved the best subset selection problem by branch-and-bound. A classic modeling practice is the so-called big-M reformulation \cite{glover1975improved}, which replaces the nonconvex complementarity constraints \eqref{MIQP:const3} with the constraints $-M x_i \leq y_i \leq M x_i, \; \forall i = 1,\dots,n$. The big-M value is chosen so that the true optimal solution is not eliminated; on the other hand, the strength of the big-M reformulation relies heavily on the value of $M$, and a loose $M$ will result in a  weak relaxation of MIQP.  MIQP formulated by big-M constraints can be equivalently written as mixed-integer second-order conic programs. \citet{vielma2008lifted} solve the resulting formulations using lifted polyhedral approximations for second-order cones. Another line of branch-and-bound methods builds upon the perspective reformulation. \citet{Frangioni2006} solve MIQP using perspective reformulation and a nonlinear branch-and-bound scheme. Subsequently, \citet{Frangioni2007} and \citet{zheng2014improving} extend the work of \citet{Frangioni2006} by designing different approaches for diagonal decomposition, and \cite{frangioni2016approximated} propose an approximate perspective reformulation through a project-and-lift approach. Another line of work accelerates the branch-and-bound procedure through specialized nonlinear programming algorithms with `warm start' capability \cite{bertsimas2009algorithm, hazimeh2022sparse}.

\vspace{1em}

\noindent \textbf{Outer approximation}: The OA algorithm \cite{duran1986outer,fletcher1994solving} solves general mixed-integer nonlinear programming (MINLP) programs to optimality. It consists of recursively solving a master mixed-integer linear (MILP) approximation of MINLP and nonlinear subproblems to generate cutting planes. The advantages of OA methods compared with nonlinear branch-and-bound methods are (i) the reduced effort of solving linear programming relaxations than nonlinear programming relaxations, (ii) the ability to utilize the `warm start' capability of linear programming solvers, the heuristics of MILP solvers, and MILP cuts. On the other hand, the strength of the cutting planes employed is crucial to the success of OA methods. In early work \cite{borchers1997computational,fletcher1998numerical}, the computational results show that nonlinear branch-and-bound outperforms OA for MIQP. The initiative to redesign OA to solve MIQP faster was taken by \citet{fischetti2017redesigning}, in which an OA algorithm using the cutting planes of a perspective reformulation of the quadratic uncapacitated facility location problem \cite{Gunluk2010} was proposed. Since then, significant advances have been accelerating the OA algorithm to solve MIQP. For example, \citet{bertsimas2020sparse} propose an OA algorithm based on the perspective reformulation of the best subset selection problem with an $\ell_{2}$-norm penalty. \citet{bertsimas2022scalable} extend the work of \citet{bertsimas2020sparse} and give an OA algorithm for solving MIQP in its general form; however, their algorithm also alters the original problem by introducing an $\ell_{2}$-norm penalty; and the cutting planes there are generated by solving a quadratic programming (QP) problem of size $n$, which could be computationally expensive. \citet{friedrichprojective} remedy this issue by replacing constraints \eqref{MIQP:const1} with a quadratic penalty term into the objective, which allows the use of the more efficient cut generation scheme in \cite{bertsimas2020sparse} (for unconstrained MIQP), but their approach does not handle the linear linking constraints \eqref{MIQP:const2}. The cutting planes used in \cite{fischetti2017redesigning, bertsimas2020sparse, bertsimas2022scalable, friedrichprojective} are `projective' cuts in the $(x, \eta)$ space, and they are derived from the subgradient cuts of the marginal function of a convex envelope of $\mathcal{Z}$, i.e., a closed convex function $f_{\text{env}}$ such that $\mathcal{Z} \subset  \{(x,y,\eta) : \eta \geq f_{\text{env}}(x, y)\} = \text{epi}(f_{\text{env}})$. As a result, the master problem in the OA algorithm is a mixed-integer linear programming problem over the binary vector $x$, and an auxiliary continuous variable $\eta$ capturing the objective function value. The OA approach using projective cutting planes has several algorithmic advantages. First, the problem size is reduced from $2n$ to $n$. Second, MIQP problems often have cardinality constraints on $x$, and such sparsity can be leveraged by commercial MILP solvers.

In the present paper, we fill the gap between the convexification theory of MIQP and the design of algorithms to solve MIQP.  We develop a unifying framework to derive projective cutting planes in $(x, \eta)$ space based on different reformulations of MIQP.  As a special case, we give a formula for computing the projective cutting planes for the perspective reformulation of MIQP. Our formula is more efficient than \cite{bertsimas2020sparse} and can handle MIQP with linear linking constraints \eqref{MIQP:const2}.

\subsection{Structure}
The rest of  this paper is organized as follows:

\noindent $\bullet$ In \S \ref{sec:1}, we first lay the theoretical foundation for our approach by appealing to variational analysis. Later, we derive cutting planes for the perspective reformulation of MIQP. A key observation is that the cutting planes can be computed by solving a QP (with a size much smaller than $n$ when a strong sparsity constraint is enforced) plus a few linear algebraic operations.

\noindent $\bullet$ In \S \ref{sec:2}, we derive cutting planes based on the perspective reformulation strengthened by the conic-quadratic inequalities of rank-one quadratics \cite{atamturk2019rank}. In addition, we show rigorously that we can obtain a class of stronger cutting planes with rank-one strengthening. 

\noindent $\bullet$ In \S \ref{sec:3}, we design an OA procedure with the cutting planes derived in \S \ref{sec:1} and \S \ref{sec:2} and apply it to the sparse portfolio selection problem studied in \cite{Frangioni2006, bertsimas2022scalable}. The computational results reveal that our algorithms lead to significant speedups compared with other state-of-the-art methods in the literature.

\section{Notation}
We denote the vector of all ones by $\mathbf{1}$. For a matrix $H$, we use the upper case $H_{i}$ to denote the $i$th column of $H$ and lower case $h_{i}$ to denote the $i$th row of $H$. We use $\succcurlyeq$ to denote the partial order in the space of symmetric matrices defined via the cone of positive semidefinite matrices. For a convex set $\mathcal C$, we adopt the conventional notations $\text{cl}(\mathcal C)$ and $\text{ri}(\mathcal C)$ to denote its closure and relative interior. The closure of the convex hull of $\mathcal C$ is written as $\clconv(\mathcal C)$. For a convex function $f(x)$, its epigraph is $\text{epi}(f) = \{(x,\eta):\eta \geq f(x)\}$, and its Fenchel conjugate is $f^{*}(\mu) = \sup_{x} \{ \langle \mu, x \rangle - f(x)\}$. The effective domain of $f$ is denoted as $\text{dom}(f) = \{x \; | \; f(x) < +\infty\}$. The subdifferential set of $f$ at $x$ is written as $\partial f(x)$. We define the indicator function of set $\mathcal C$ as $\mathcal{I}_{\mathcal C}(x) = 0$ if $x \in \mathcal C$ and $+\infty$ otherwise. We define the projection operator $\text{Proj}_{x, \eta}(\mathcal C)$ as $\text{Proj}_{x, \eta}(\mathcal C) = \{(x, \eta) : (x, y, \eta) \in \mathcal C \; \text{for some} \; y\}$.

\section{Cutting Planes for Perspective Reformulation}\label{sec:1}
For an arbitrary closed convex envelope $f_{\text{env}}$ of $\mathcal{Z}$, we have, by definition, $\mathcal{Z} \subset \text{epi}(f_{\text{env}})$. It is known from convex analysis that the marginal function $\bar{f}_{\text{env}}$ obtained by partially minimizing over $y \in \R^{n}$, i.e., 
$$
\bar{f}_{\text{env}}(x) = \inf_{y} f_{\text{env}}(x, y),
$$
is also convex. 

\begin{lemma}\label{lem:new}
For any $t \in \partial \bar{f}_{\text{env}}(x)$, the following inequality 
\begin{align}\label{subgradientcut}
\tilde \eta \geq \langle t, \tilde x - x \rangle + \bar{f}_{\text{env}}(x)    
\end{align}
holds for any $(\tilde x, \tilde y, \tilde \eta) \in \clconv(\mathcal{Z})$.
\end{lemma}
\begin{proof}
Because $\bar{f}_{\text{env}}$ is a convex function, the subgradient inequality \eqref{subgradientcut} is valid for $\text{epi}(\bar{f}_{\text{env}})$. Moreover, we have the sequence of inclusion relations that $\text{Proj}_{x, \eta}(\clconv(\mathcal{Z})) \subset \text{Proj}_{x, \eta}(\text{epi}(f_{\text{env}})) \subset \text{epi}(\bar{f}_{\text{env}})$. The first inclusion follows from the fact that $\mathcal{Z} \subset \text{epi}(f_{\text{env}})$ and $\clconv(\mathcal{Z}) \subset \text{epi}(f_{\text{env}})$ because $\text{epi}(f_{\text{env}})$ is a convex set. For any $y$, $\eta \geq f_{\text{env}}(x, y)\ge \bar{f}_{\text{env}}(x)= \inf_{y \in \R^{n}} f_{\text{env}}(x, y)$, 
 by definition, thus the second inclusion holds. Based on the inclusion relations, we can see that the subgradient inequality \eqref{subgradientcut} is valid for $\text{Proj}_{x, \eta}(\clconv(\mathcal{Z}))$, and we write the projection in an explicit form as $\text{Proj}_{x, \eta}(\clconv(\mathcal{Z})) = \{(x,\eta): \; (x, y, \eta) \in \clconv(\mathcal Z) \; \text{for some} \; y \in \R^{n} \}$. Now, it becomes clear that the subgradient inequality is also valid for $\clconv(\mathcal{Z})$.  
\end{proof}

One insight that can be drawn from Lemma ~\ref{lem:new} is that we can construct projective cutting planes for $\mathcal{Z}$ using the subgradients of $\bar{f}_{\text{env}}$. Our goal is to design a unifying procedure for computing the subgradients regardless of the explicit form of $\bar{f}_{\text{env}}$, so that we can derive a rich family of projective cutting planes via various convex envelopes available in the literature. Toward this end, we will borrow tools from variational analysis. 

In what follows, for a proper convex function $f$ whose definition is not clear over the relative boundary of $\text{dom}(f)$, we assume that we naturally extend it to $\text{cl}(\text{dom}(f))$ by the limiting argument
\begin{equation*}
f(x) = \lim_{\lambda \to 0^{+}} f(x + \lambda(\tilde x - x)) \quad \quad \text{for some $\tilde x \in \text{ri}(\text{dom}(f))$}.
\end{equation*}
For example, for $\phi(x, y) = \frac{y^2}{x}$,
\begin{equation*}
\phi(x, y) = \left\{\begin{array}{cc}
   \frac{y^2}{x}  & \text{if $x > 0$}  \\
   0  & \text{if $y = x = 0$} \\
   +\infty & \text{otherwise.}
\end{array}
\right.
\end{equation*}

We also assume that problem \eqref{eqn:MIQP} is feasible. 
\begin{assumption}\label{assumption:1}
There exists $(\tilde x, \tilde y)$ satisfying constraints \eqref{MIQP:const1}--\eqref{MIQP:const4}.
\end{assumption}

Checking the feasibility of the system \eqref{MIQP:const1}--\eqref{MIQP:const4} is $\mathcal{NP}$-hard in general because it includes the feasibility problem of a mixed-integer linear set as a special case. However, for many interesting applications, i.e., sparse regression, best subset selection, and sparse portfolio selection, it is without loss of generality to assume that the MIQP problem always has a feasible point. 
\begin{lemma}\label{lem:1}
For $f(x, y) = \frac{y^2}{x} \; : \R^{+} \times \R \to \R^{+}$, its subdifferential $\partial f(x, y)$ is as follows

\begin{equation*}
\partial f(x, y) = \left\{\begin{array}{cc}
   \{(\Phi, \Psi) \; | \; \Phi \leq -\frac{1}{4}\Psi^2\}   & \text{if} \;  x = 0 \; \text{and} \; y = 0\\
    (- \frac{y^2}{x^2}, \frac{2y}{x}, ) & \text{if} \; x >0 \\
    \emptyset & \text{otherwise.}
\end{array}
\right.
\end{equation*}
\end{lemma}

\begin{proof}
The result follows from the Fenchel-Young inequality \cite{rockafellar1970convex}:
\begin{equation*}
    f(x,y) + f^{*}(\Phi,\Psi) \geq x \Phi + y \psi
\end{equation*}
and the inequality holds at equality if and only if
\begin{equation*}
(\Phi, \Psi) \in \partial f(x, y).
\end{equation*}
The conjugate function $f^{*}(\Phi,\Psi)$, by definition, equals 
\begin{equation*}
    f^{*}(\Phi,\Psi) = \sup_{x,y} \left\{ \Phi x + \Psi y - \frac{y^2}{x}\right\} = \mathcal{I}_{\{(\Phi,\Psi)\;|\;\Phi \leq -\frac{1}{4} \Psi^2\}}(\Phi, \Psi).
\end{equation*}
Note that the perspective function $f$ is differentiable in its relative interior $x > 0$, and the subdifferential collapses to a single point, and when $x = 0$ but $y \neq 0$, $f(0, y) = +\infty$, so no subgradient exists. The only nontrivial case is when $x = 0$ and $y= 0$. By Fenchel-Young inequality, we know that $(\Phi, \Psi) \in \partial f(0,0)$ if and only if
\begin{align*}
    \mathcal{I}_{\{(\Phi,\Psi)\;|\;\Phi \leq-\frac{1}{4}\Psi^2\}}(\Phi,\Psi) = 0 \; \iff \; \Phi \leq -\frac{1}{4} \Psi^2.
\end{align*}
\end{proof}

Before we proceed with deriving the first class of cutting planes for $\mathcal Z$, we recall a result regarding the subdifferential of a marginal function from variational analysis, which is our main tool for computing projective cutting planes.
\begin{theorem}{\cite{rockafellar2009variational}}\label{thm:1}
For any proper closed convex function $f(x,y) : \R^{2n} \rightarrow \R \bigcup \{\infty\}$, 
suppose that there exists $\bar x, \bar r$ such that the level set
\begin{equation*}
    \{y \; | \; f(\bar x, y) \leq \bar r\}
\end{equation*}
is compact, then 
\begin{align*}
(\textit{i}) \quad \quad \quad  & \bar f(\cdot): \R^{n} \rightarrow \R \; \text{is finite and lower semicontinuous,} \\
(\textit{ii}) \quad \quad \quad &
\partial \bar{f}(x) = \{t \; : \; (t, 0) \in \partial f(x, y) \; \text{for some $y \in \R^{n}$}\}.
\end{align*}
\end{theorem}
Here, we provide a simplification of the condition of Theorem 10.31 in \cite{rockafellar2009variational} because when $f$ is a closed convex function, the boundedness of the level set $\{y \; | \; f(\bar x, y) \leq \bar r \}$ at some point $\bar x$ for some value $\bar r$  implies the uniform level boundedness (Definition 1.16 \cite{rockafellar2009variational}).

The level set boundedness condition in Theorem ~\ref{thm:1} is necessary to guarantee the validity of the subdifferential formula in Theorem ~\ref{thm:1} for the marginal function $\bar f(\cdot)$.
Next, we will see an example in which the formula does not hold in the absence of the level set boundedness condition. 

\begin{example}\label{ex:1}
We define function $f_1(x, y)\; : \R^{+} \times \R \to \R^{+}$  piecewise  as follows:
\begin{equation*}
f_1(x, y) = \left\{
\begin{array}{cc}
x \exp{(-\frac{y}{x})}     & \text{$x > 0$} \\
0    & \text{$x = 0, y = 0$} \\
0    & \text{$x = 0, y > 0$} \\
+\infty & \text{$x = 0, y < 0$}.
\end{array}
\right.
\end{equation*}
\end{example}
Note that on $\text{ri}(\text{dom}(f_1))$, $f_1$ coincides with the perspective function of a simple exponential, and on the boundary, it is not hard to check that the function values can be obtained via a limiting argument. Thus, $f_1$ is a closed convex function. It is easy to see that the level set boundedness condition does not hold for $f_1$. The marginal function $\bar f_1$ after partially minimizing over $y$ is
\begin{equation*}
\bar f_{1}(x) = 0,
\end{equation*}
which is a constant function over $[0, + \infty]$. Therefore, $\bar f_{1}$ is also a closed convex function whose subdifferential set $\partial \bar f_1(x)$ is nonempty and compact for any $x \geq 0$; however, a simple calculation reveals that $\partial f_1(x, y) = \{(\exp(-\frac{y}{x}) + \frac{y^2}{x}\exp(-\frac{y}{x}), - \exp(-\frac{y}{x})) \}$ for $x > 0$, and the second coordinate $- \exp(-\frac{y}{x})$ is always nonzero. This means there does not exist a point $y$ such that $(t, 0) \in \partial f_1(x, y)$ for any $x > 0$. Figure~\ref{fig:1}  provides a visualization of $f_1$.
\begin{figure}[h]
    \centering
    \includegraphics[width = 0.5\textwidth]{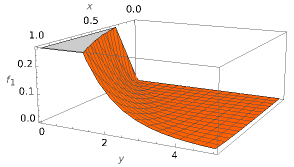}
    \caption{$f_{1}$.}
    \label{fig:1}
\end{figure}

We introduce the notations $\mathit{F}_{1}$ and $\mathit{F}_{2}$ to denote the sets of feasible points of the linear inequalities \eqref{MIQP:const1} and \eqref{MIQP:const2}, respectively, i.e., $\mathit{F}_{1} = \{ y: A y \leq b\}$ and $\mathit{F}_{2} = \{(x,y): Cy \leq D x\}$.  
Starting from the perspective reformulation, we first give a corresponding convex envelope function of $\mathcal Z$. Suppose we decompose $Q$ as $Q = \text{diag}(\delta) + R$ such that $\delta \geq 0$ and $R \succcurlyeq 0$. The perspective reformulation is as follows:
\begin{align*}
\min_{x, y} \quad & y^{\top} R y + \sum_{i = 1}^{n} \delta_i \frac{y^2}{x_i} + g^{\top} y + h^{\top} x \\
\text{s.t.} \quad & A y \leq b \\
                & C y \leq D z \\
                & x \in X.
\end{align*}
Alternatively, we can make use of the indicator function, and replace the constraints on the continuous variables with indicator functions in the objective function, i.e.
\begin{align*}
\min_{x, y} \quad & y^{\top} R y + \sum_{i = 1}^{n} \delta_i \frac{y^2}{x_i} + g^{\top} y + h^{\top} x + \mathcal{I}_{\mathit F_1}(y) + \mathcal{I}_{\mathit F_2}(x, y) \\
\text{s.t.} \quad & x \in X.
\end{align*}
Let us denote the objective function above as $f_{\text{persp}}(x,y)$. Note that the perspective relaxation is obtained when we replace the binary restriction $x \in X$ with $x \in \conv(X)$. We have the relation that  $(x, y, \eta) \in \mathcal Z \iff \eta \geq f_{\text{persp}}(x, y)$ and $x \in X$ (if $x_i$ is zero for some $i \in [n]$, then to avoid $f_{\text{persp}}$ being infinity, we must have $y_i$ equal zero as well, and the complementarity constraints are satisfied), thus $f_{\text{persp}}$ is a valid convex envelope function of $\mathcal Z$.  Here, we ignore the subtle case where $\delta_i$ may be zero for some $i \in [n]$ since it is not hard to see that $f_{\text{persp}}$  is a convex envelope function of $\mathcal{Z}$ under this scenario as well. As we have argued before, for any $x \in \conv(X)$ and  $t \in \partial \bar{f}_{\text{persp}}(x)$, the  subgradient inequality $\tilde \eta \geq \langle t, \tilde x - x \rangle + \bar{f}_{\text{persp}}(x)$ is valid for any $(\tilde x, \tilde y, \tilde \eta) \in \mathcal{Z}$. In Proposition ~\ref{prop:1}, we will characterize $\partial \bar{f}_{\text{persp}}$.
\begin{proposition}\label{prop:1}
For any $x \in [0,1]^{n}$, suppose that $S$ is the set of indices of nonzero entries of $x$, and $S^{C}$ denotes its complement. Let $A_{S}$ and $C_{S}$ denote the submatrices where the columns are in $S$. For $ y \in \R^{n}, \lambda \in \R^{m_1}$ and $\mu \in \R^{m_2}$, let $y_{i} = 0$ for $i \in S^{C}$ and  $y_S$, $\lambda$, and $\mu$ be the primal-dual optimal solution of the following quadratic program (QP) in the reduced space indexed by $S$ 
\begin{subequations}\label{prop1:QP}
\begin{align}
\min_{\tilde y \in \R^{|S|}} \quad & \tilde y^{\top} \left(R_{S} + \text{diag}\left(\left\{\frac{\delta_i}{x_i}\right\}_{i \in S}\right)\right) \tilde y + g_{S}^{\top} \tilde y \\
\text{s.t.} \quad & A_{S} \tilde y \leq b \\
            & C_{S} \tilde y \leq D x. 
\end{align}
\end{subequations}
Then $t \in \partial \bar f_{\text{persp}}( x)$ if and only if there exists $\Psi \in \R^{|S^{C}|}$ such that
\begin{align}
& t_i = - \delta_i \frac{y_i^2}{x_i^2} - \mu^{\top} D_i + h_i, \quad \forall i \in S \label{prop1:statement1}\\
& \delta_i \Psi_i = - (2R_i^{\top}  y + g_i + \lambda^{\top} A_i + \mu^{\top} C_i), \quad \forall i \in S^{C} \label{prop1:statement2}\\
& t_i \leq -\frac{1}{4} \delta_i \Psi_i^2 - \mu^{\top} D_i + h_{i}, \quad \forall i \in S^{C}. \label{prop1:statement3}
\end{align}
\end{proposition}

\begin{proof}
 Given $ x\in [0,1]^{n}$, let $S$ denote the support set of $x$ (i.e., $S = \{ i \in [n] \; | \; x_i \neq 0\}$). By Assumption ~\ref{assumption:1}, there exists a feasible point $(\tilde x, \tilde y) \in X \times \R^{n}$ for \eqref{eqn:MIQP}.  We have the inequality $f_{\text{persp}}(\tilde x, y) \geq y^{\top} Q y + g^{\top} y + h^{\top} \tilde x$ since $\delta_i \frac{y_{i}^{2}}{\tilde{x}_i} \geq \delta_i y_i^{2}$ for any $\tilde{x}_i \in [0,1]$. The strongly convex quadratic function $y^{\top} Q y + a^{\top} y$ has bounded level sets, which implies that any level set of $f_{\text{persp}}(\tilde x, \cdot)$ as a function of $y$ is also bounded.  Note, in particular, that the level set $\{y \; | \; f_{\text{persp}}(\tilde x, y) \leq f_{\text{persp}}(\tilde x, \tilde y) \}$ is nonempty as it contains $\tilde y$, and it satisfies the level set boundedness condition in Theorem ~\ref{thm:1}.  Since we extend the perspective function $\frac{y_i^2}{x_i}$ to its closure, we have $\frac{y_i^2}{0} = +\infty$ if $y_i \neq 0$.  Then for the subdifferential $\partial f_{\text{persp}}(x,y)$ to exist for some $y \in \R^{n}$, we must have $y_i = 0$ for any $i \in S^{C}$ and $(x,y)$ is feasible for \eqref{eqn:MIQP} (one of the indicator functions $\mathcal{I}_{\mathit{F}_{1}}$ or $\mathcal{I}_{\mathit{F}_{2}}$ will evaluate to $+\infty$ if otherwise). 
Recall that the subdifferential of the indicator function $\mathcal I_{\mathcal C}(\cdot)$ of a closed convex set is its normal cone $\mathcal N_{\mathcal C}(\cdot)$ \cite{Hiriart2013}, and the normal cones of polyhedral sets $\mathit F_{1}$ and $\mathit F_{2}$ are:
\begin{align*}
& \mathcal N_{\mathit{F}_{1}} (y)  = \{(0, \tilde{\lambda}^{\top} A) \; | \; \tilde{\lambda} \geq 0, \tilde{\lambda}_{i} = 0, \; \forall i \text{ such that} \; \langle a_i, y \rangle < b_i\},   \\
& \mathcal N_{\mathit{F}_{2}} (x,y) = \{ (- \tilde{\mu}^{\top} D, \tilde{\mu}^{\top} C) \; | \; \tilde{\mu} \geq 0, \tilde{\mu}_{i} = 0, \; \forall i \text{ such that} \; \langle c_i, y \rangle < \langle d_i, x \rangle \}. 
\end{align*}
The indicator function of a polyhedron is a polyhedral function \cite{rockafellar1970convex},  and by Theorem 23.8 in \cite{rockafellar1970convex} (in which the constraint qualification holds), $\partial f_{\text{persp}}(\tilde x, \tilde y)$ equals the sum of the subdifferentials of the components of $f_{\text{persp}}$, i.e., $y^{\top} R y$, $\delta_i \frac{y_i^2}{x_i}$, $\mathcal{I}_{\mathit{F}_{1}}$, and $\mathcal{I}_{\mathit{F}_{2}}$.  By Theorem ~\ref{thm:1}, $t \in \partial \bar{f}_{\text{persp}}(x)$ if and only if $t$ solves the first-order equation for some $y \in \R^{n}$:
$$
(t, 0) \in \partial f_{\text{persp}}(x, y).
$$

We will reveal the equation above entry-by-entry in the order of (i) $y_i$ for $i \in S$, (ii) $y_i$ for $i \in S^{C}$, (iii) $x_i$ for $i \in S$, and (iv) $x_i$ for $i \in S^{C}$ . 

\noindent The first set of equations is for $i \in S$:
\begin{align}\label{prop1:eqn1-1}
&2R_{i}^{\top} y + 2\delta_i \frac{y_i}{x_i} + g_i + \tilde{\lambda}^{\top} A_i + \tilde{\mu}^{\top} C_i = 0. 
\end{align}
The second set of equations is that there exists $\Psi \in \R^{|S^{C}|}$, where $\Psi_i$ is the second entry of the subgradients in $\partial  \frac{y_i^2}{x_i}$ at $(0,0)$, such that for $i \in S^{C}$:
\begin{align}\label{prop1:eqn1-2}
&\delta_i \Psi_i  + 2R_i^{\top} y + g_i + \tilde{\lambda}^{\top} A_i + \tilde{\mu}^{\top} C_i = 0,
\end{align}
which is nothing but \eqref{prop1:statement2}.
The third set of equations is for $i \in S$:
\begin{align}\label{prop1:eqn1-3}
&t_i = - \delta_i \frac{ y_i^2}{x_i^2} - \tilde{\mu}^{\top} D_i + h_{i},    
\end{align}
which is nothing but \eqref{prop1:statement1}.
The last set of equations is that there exists $\Phi \in \R^{|S^{C}|}$, where $\Phi_i$ is the first entry of the subgradients in $\partial \frac{y_i^2}{x_i}$ at $(0,0)$, such that for $i \in S^{C}$:
\begin{align}\label{prop1:eqn1-4}
t_i = -  \delta_i \Phi_i - \tilde{\mu}^{\top} D_i + h_i.
\end{align}    
In equations \eqref{prop1:eqn1-1}--\eqref{prop1:eqn1-4}, by Lemma ~\ref{lem:1}, $\Psi_i, \Phi_i$ for $i \in S^{C}$ satisfy the relations:
\begin{align}\label{prop1:eqn1-5}
\Phi_i \leq - \frac{1}{4} \Psi_i^2, \quad \forall i \in S^{C}.
\end{align}

Thus far, we have argued that equations \eqref{prop1:eqn1-1}--\eqref{prop1:eqn1-5} must hold, $(x, y)$ must be feasible to \eqref{eqn:MIQP}, and that $y_i = 0$ for all $i \in S^{C}$ for $\partial f_{\text{persp}}(x, y)$ to exist. Next, we simplify these equations to obtain the desired result. Note that equations \eqref{prop1:eqn1-1} where $\tilde{\lambda}$ and $\tilde{\mu}$ satisfy the complementarity condition together with the conditions that the given $(x,  y)$ satisfies constraints \eqref{MIQP:const1}--\eqref{MIQP:const2} and $y_i = 0$ for any $i \in S^{C}$ are equivalent to the  Karush–Kuhn–Tucker (KKT) conditions of the QP problem \eqref{prop1:QP}. Hence,  $(y_{S}, \tilde \lambda, \tilde \mu)$ is the primal-dual optimal solution for the QP problem \eqref{prop1:QP}. Furthermore, we can combine equations \eqref{prop1:eqn1-4} and \eqref{prop1:eqn1-5} by plugging \eqref{prop1:eqn1-5} into \eqref{prop1:eqn1-4} which then gives \eqref{prop1:statement3}.

\end{proof}

By substituting \eqref{prop1:statement2} into \eqref{prop1:statement3}, we can see that we can choose any $t$ satisfying \eqref{prop1:statement1} and the following inequality: 
\begin{equation}\label{prop1:eqn1-6}
t_i \leq -\frac{1}{4\delta_i}(2R_i^{\top}  y  + g_i + \lambda^{\top} A_i + \mu^{\top} C_i)^2 - \mu^{\top} D_{i} + h_i, \quad \forall i \in S^{C},
\end{equation} in deriving the subgradient cuts for $\mathcal{Z}$. Nevertheless, 
the strongest cuts are those that attain equality. To see this, if we add the cut 
\begin{equation*}
   \tilde \eta \geq \langle t, \tilde x - x \rangle + \bar f_{\text{persp}}(x),
\end{equation*}
where $x_i = 0, \; \forall i \in S^{C}$, and expand it as $\tilde \eta \geq \sum_{i \in S} t_i (\tilde  x_i - x_i) + \sum_{i \in S^{C}} t_i \tilde x_i + \bar f_{\text{persp}}(x)$, then, a cut with smaller $t_i$ for $i\in S^{C}$ is implied by the one with  larger $t_i$ for $i \in S^{C}$ and the bounds $0 \geq -x_i$ for $i \in [n]$. In Algorithm ~\ref{alg:3}, we outline the procedure for obtaining the strongest cutting plane from the perspective reformulation.

\begin{algorithm}[H]
\caption{A cut-generating procedure based on the perspective reformulation ($\text{CUT}_{\text{persp}}$)} \label{alg:3}
\begin{algorithmic}[1]
\STATE \textbf{INPUT} $R, \delta, g,h, (A,b), (C,D), x$
\STATE Set $S$ to be the set of indices such that $x_i > 0$
\STATE Set $y = 0$
\STATE Set $(y_{S}, \lambda, \mu)$ to be the optimal primal-dual solution of QP 
\eqref{prop1:QP}
\STATE Set $t = 0$
\FOR{$i \in S$}
\STATE Set $t_i = - \delta_i \frac{y_i^2}{x_i^2} - \mu^{\top} D_{i} + h_i$ 
\ENDFOR
\FOR{$i \in S^{C}$}
\STATE Set $\Psi_i = - \frac{1}{\delta_i}(2\sum_{j \in S} R_{ij} y_j + g_i + \lambda^{\top} A_i + \mu^{\top} C_i)$ 
\ENDFOR
\FOR{$i \in S^{C}$}
\STATE Set $t_i = - \frac{\delta_i}{4} \Psi_i^2 - \mu^{\top} D_{i} + h_i$
\ENDFOR
\STATE \textbf{return} $t$
\end{algorithmic}
\end{algorithm}

In some cases, e.g., in the presence of bound constraints, some constraints active at $y$ may not involve $y_i$ for $i \in S$, and the corresponding multipliers $ \lambda_i$ and $ \mu_i$ do not appear in the optimal primal-dual solution of the QP \eqref{prop1:QP}. Consequently, such $\lambda_i$ and  $\mu_i$  are `free'  variables.  As we have argued before, it would be beneficial to maximize the right-hand side of \eqref{prop1:eqn1-6}; however, doing so will require solving a QP.  One simple strategy is to set these $\lambda_i$ and $\mu_i$ to zero, which is how we implement the cutting planes in the numerical experiments.

One insight we can draw from Proposition~\ref{prop:1} is that computing a cutting plane can be very efficient if only a small proportion of the entries of $x$ are nonzero. When we enforce a sparsity constraint, i.e., $X \subseteq X_{k}$, at any binary point, the computational cost equals the solution time of a $k$-dimensional quadratic program \eqref{prop1:QP} plus $\mathcal{O}(n(k + m_1 + m_2))$  floating-point operations. Next, we give a corollary of Proposition ~\ref{prop:1} for the case in which we do not have constraints over $y$.
\begin{corollary}\label{corollary:1}
For MIQP with no constraints over $y$, given $x \in [0,1]^{n}$, suppose $S$ is the set of indices of nonzero entries of $x$; for $y \in \R^{n}$, let $ y_{i} = 0$ for $i \in S^{C}$ and  $y_{S}$ be the unconstrained optimal solution in the reduced subspace indexed by $S$, i.e., $ y _{S} = -\frac{1}{2} Q_{S}^{-1} g_{S}$. Then $t \in \partial \bar f_{perps}(x)$ if and only if
\begin{align*}
    & t_i = - \delta_i \frac{ y_i^2}{x_i^2} + h_i\quad i \in S\\
    & \delta_i \Psi_i = - (2 R_i^{\top}  y + g_i) \quad \forall i \in S^{C} \\
    & t_i \leq - \frac{1}{4} \delta_i \Psi_i^2 + h_i \quad \forall i \in S^{C}.
\end{align*}
\end{corollary}
\begin{proof}
The proof follows the same procedure as the proof of Proposition ~\ref{prop:1} by ignoring the constraints \eqref{MIQP:const1} and \eqref{MIQP:const2}. 
\end{proof}

Corollary ~\ref{corollary:1} gives a formula for computing the subgradients of $\bar f_{\text{persp}}(x)$ in the absence of constraints on $y$ for an arbitrary diagonal decomposition ($Q = R + \text{diag}(\delta)$ such that $R \succcurlyeq 0$ and $\delta \geq 0$), and it generalizes the result in \cite{bertsimas2020sparse} by allowing the remainder term $R$ to be rank-deficient. 

\citet{bertsimas2022scalable} proposed a class of cutting planes for $\mathcal Z$ when $X = X_{k}$ based on a convex Boolean reformulation of \eqref{eqn:MIQP}. Their result requires first reformulating \eqref{eqn:MIQP} equivalently as a sparse regression problem, i.e.,
\begin{subequations}\label{eqn:sparseregression}
\begin{align}
\min_{x,y} \quad & \sum_{i = 1}^{n} \delta_i y_i^2 + \| \beta - E y \|_{2}^{2} + \tilde{g}^{\top} y \\
\text{s.t.} \quad & A y \leq b \\
                  & C y \leq D x \\
                  &y_i(1 - x_i) = 0, \quad \forall i = 1,\dots,n \\
                  &x \in X_k,
\end{align}
\end{subequations}
where $R = E^{\top}E$ for some full row-rank $E \in \R^{k_1 \times n}$, $\beta = -\frac{1}{2}(EE^{\top})^{-1}E g$, and $\tilde{g} = (I - E^{\top}(EE^{\top})^{-1}E)g$, and then applying the convex Boolean reformulation akin to \cite{pilanci2015sparse, bertsimas2020sparse}.
The perspective reformulation of \eqref{eqn:sparseregression} is written as:
\begin{subequations}\label{eqn:sparsepersp}
\begin{align}
\min_{x,y} \quad  & \sum_{i = 1}^{n}  \delta_i \frac{y_i^2}{x_i} + \|r\|_{2}^{2} + \tilde{g}^{\top} y + c^{\top} x\\
\text{s.t.} \quad & r = \beta - Ey \quad \\
                  & A y \leq b \quad \quad \quad \\
                  & C y \leq D z \quad \quad \\
                  &x \in X_k.
\end{align}
\end{subequations}
We introduce another indicator function $\mathcal I_{\mathit F_{3}}(\cdot)$ with $\mathit{F}_{3} \equiv \{(y,r):r = \beta - Xy\}$, and define $$f_{\text{persp}_{2}}(x, y, r) = \sum_{i = 1}^{n} \delta_i \frac{y_i^2}{x_i} + \|r\|_{2}^{2} + \tilde{g}^{\top} y + c^{\top} x + \mathcal I_{\mathit{F}_{3}}(y, r) + \mathcal I_{\mathit{F}_{1}}(y) + \mathcal I_{\mathit{F}_{2}}(x, y).$$ 
In  Appendix \ref{app:equiv}, we prove that  $\partial \bar f_{\text{persp}_{2}}(x)$ coincides with the cutting planes in  \cite{bertsimas2022scalable} and since $f_{\text{persp}_{2}}$ only differs from $f_{\text{persp}}$  by a constant, $\partial \bar f_{\text{persp}} = \partial \bar f_{\text{persp}_{2}}$.  However, the cost of computing a cutting plane in \cite{bertsimas2022scalable} includes the cost of solving a QP of size $k_1$ plus $\mathcal O(n(k_1 + m_1 + m_2))$ floating-point operations, where $k_1$ (the rank of $R$) can be of the same order as $n$; on the contrary, the formula in Proposition ~\ref{prop:1} only requires the solution of a QP of size $k$ in addition to $\mathcal O(n(k + m_1 + m_2))$ floating-point operations. Observe that $k$ can be much smaller than $n$ in some applications. 

\begin{remark}\label{rem:6}
At any binary point $x \in X$, $\bar{f}_{\text{persp}}(x)$ is the optimal value of MIQP \eqref{eqn:MIQP} given $x$.
\end{remark}
When computing $\bar{f}_{\text{persp}}(x) = \min_{y \in \R^{n}} f_{\text{persp}}(x, y)$, to avoid the term $\delta_i \frac{y_i^2}{x_i}$ becoming $+\infty$, we must have $y_i = 0$ if $x_i = 0$ and thus the complementarity constraint \eqref{MIQP:const3} holds. On the other hand, for any $y$ satisfying the complementarity constraint \eqref{MIQP:const3},  $y^{\top} \left(\text{diag}\left(\left\{\frac{\delta_i}{x_i}\right\}_{i \in [n]}\right) + R \right) y = y_{S}^{\top} \left(\text{diag}(\delta_{S}) + R_{S} \right) y_{S} = y_{S}^{\top} Q_{S} y_{S}$ since the term $\delta_i \frac{y_i^2}{x_i} = \frac{0}{0} = 0$ for any $i \in S^{C}$. Hence, the partial minimization problem $\min_{y \in \R^{n}} f_{\text{persp}}(x, y)$ is equivalent to
\begin{align*}
\min_{\tilde y \in \R^{|S|}} \quad & \tilde y^{\top} Q_{S} \tilde y + g_{S}^{\top} \tilde{y} + h^{\top} x \\
\text{s.t.} \quad & A_{S} \tilde y \leq b \\
                & C_{S} \tilde y \leq D x,
\end{align*}
which is exactly MIQP \eqref{eqn:MIQP} given $x$. 

So far, we have seen how to derive cutting planes for $\mathcal Z$ in the space of $(x, \eta)$ using Theorem ~\ref{thm:1} and the perspective reformulation. It is also interesting to see how we can use the same framework to derive cutting planes for $\mathcal Z$ based on other reformulations in the literature. In the next section, we focus on one reformulation that uses the ideal formulation of the epigraph of a rank-one quadratic function with indicator variables and no constraints on $y$.

\section{Cutting planes for perspective reformulation with rank-one inequalities}\label{sec:2}

Let
\begin{equation*}
    X_{R1} \equiv \{(x, y, \eta) \in \{0,1\}^{n} \times \R^{n+1} \; | \; \eta \geq (y^{\top} u)^2, \; y_i (1 - x_i) = 0, \; \forall i \in [n]\}
\end{equation*}
denote the epigraph of a rank-one quadratic function with indicator variables, and $y$ is unconstrained. \citet{atamturk2019rank} show that $\clconv(X_{R1})$ is obtained by simply adding the inequality 
\begin{equation}\label{eqn:ro}
    \eta \geq \max\left\{(u^{\top} y)^2, \frac{(u^{\top} y)^2}{\sum_{i = 1}^{n} x_i}\right\},
\end{equation}
and the closure of the convex hull of $X_{R1}$ admits the following description:
\begin{equation*}
    \clconv(X_{R1}) \equiv \left\{(x, y, \eta) \in [0,1]^{n} \times \R^{n+1} \; | \; \eta \geq \max \left\{(u^{\top} y)^2, \frac{(u^{\top} y)^2}{\sum_{i = 1}^{n} x_i}\right\}\right\}.
\end{equation*}
The validity of  \eqref{eqn:ro} follows from the fact that for any $(x, y, \eta) \in X_{R1}$ if $x_i = 1$ for some $i \in [n]$, then $\eta \geq (u^{\top} y)^2 = \max\left\{(u^{\top} y)^2, \frac{(u^{\top}y)^2}{\sum_{i = 1}^{n} x_i}\right\}$, and if $x_i = 0$ for all $i \in [n]$, then $(u^{\top} y)^2 = \frac{(u^{\top}y)^2}{\sum_{i = 1}^{n} x_i} = 0$ by the complementarity constraint. The convexity of the constraint \eqref{eqn:ro} follows from the fact that its right-hand side is the maximum of two convex functions. 

Inequality \eqref{eqn:ro} is easy to implement with a decomposition of $Q$ as the sum of rank-one quadratic matrices. In what follows, we will see that the conic-quadratic inequality \eqref{eqn:ro} together with the perspective reformulation gives a class of stronger subgradient cuts than those in Proposition~\ref{prop:1}.  

Before we derive our new cutting planes, we introduce the following lemma regarding the subdifferential of the composite of a perspective function and a linear mapping. 
\begin{lemma}\label{lem:2}
Let $\phi(x, y) = \max \left\{(u^{\top} y)^2, \frac{(u^{\top} y)^2}{\sum_{i = 1}^{n} x_i} \right\}$, then $\partial \phi(x, y) = $
\tiny{
\begin{equation*}
\left\{ 
\begin{array}{cc}
 ( 0, 2(u^{\top}y)u)  &  \sum_{i = 1}^{n} x_i > 1\\
 \conv(\{(  0, 2(u^{\top}y)u), (-(u^{\top}y)^2 \mathbf{1}, 2(u^{\top}y)u) \}) & \sum_{i = 1}^{n} x_i = 1 \\
 \left(-\frac{(u^{\top}y)^2}{(\sum_{i = 1}^{n} x_i)^2} \mathbf{1}, 2\frac{1}{\sum_{i = 1}^{n} x_i}(u^{\top}y)u\right) & 0 < \sum_{i=1}^{n} x_i < 1 \\
 \{(\Phi \mathbf{1},\Psi u) :  \Phi \leq -\frac{1}{4} \Psi^2 \} & x = 0.
\end{array}
\right.
\end{equation*}
}
\end{lemma}

\begin{proof}

There are four cases to consider:
\begin{itemize}
\item[(i)] $\sum_{i = 1}^{n} x_i > 1$. 
In a small neighborhood of such $(x, y)$, we have $\sum_{i = 1}^{n} x_i > 1$, thus $(u^{\top} y)^2 \geq \frac{(u^{\top} y)^2}{\sum_{i = 1}^{n} x_i}$. Hence, $\phi(x, y)$ equals $ (u^{\top} y)^2$ locally. Since the subdifferential is determined by the local geometry of the epigraph of $\phi(\cdot)$, $\partial \phi(x, y) = \partial (u^{\top} y)^2$.
\item[(ii)] $0 < \sum_{i = 1}^{n} x_i < 1$. 
Similarly, when $0 < \sum_{i = 1}^{n} x_i < 1$, $\phi(x, y) = \frac{(u^{\top} y)^2}{\sum_{i = 1}^{n} x_i}$ in a small neighborhood, and $\partial \phi(x, y) = \partial \frac{(u^{\top} y)^2}{\sum_{i = 1}^{n} x_i}$ which shrinks to a single point since $\frac{(u^{\top} y)^2}{\sum_{i = 1}^{n} x_i}$ is differentiable. 
\item[(iii)] $\sum_{i = 1}^{n} x_i = 1$. In this case, then $(u^{\top} y)^2$ and $\frac{(u^{\top} y)^2}{\sum_{i = 1}^{n} x_i}$ coincide, and they are finite in a small neighborhood of $(x, y)$. Then, according to the rule of the subdifferential of the supremum of finite convex functions  (Theorem 4.4.2 \cite{Hiriart2013}), $\partial \phi(x, y) = \conv(\partial (u^{\top} y)^2 \bigcup \partial \frac{(u^{\top} y)^2}{\sum_{i = 1}^{n} x_i})$.
\item[(iv)] $x = 0$. In this case, $\phi(x, y)$ is locally equal to $\frac{(u^{\top} y)^2}{\sum_{i = 1}^{n} x_i}$, the composite of a perspective function and a linear mapping. Then, $\partial \frac{(u^{\top} y)^2}{\sum_{i = 1}^{n} x_i}$ can be derived via the chain rule of subdifferentials and Lemma ~\ref{lem:1}. 
\end{itemize}
\end{proof}

In what follows, we assume that we are given a decomposition $Q = L L^{\top} + \text{diag}(\delta) + N$ such that $\delta \geq 0$ and $N \succcurlyeq 0$, and $L \in \R^{n \times K}$ for some $1\le K\le n$ (we will discuss how to obtain such a decomposition later). Then, we can apply  inequality \eqref{eqn:ro} to the terms $(L_{i}^{\top} y)^2$ for $i \in 1,\dots,K$ and arrive at the following convex envelope function for $\mathcal Z$:
\begin{align*}
f_{\text{persp+ro}}(x,y) \equiv & \sum_{i = 1}^{K} \max \left\{(L_{i}^{\top} y)^2, \frac{(L_{i}^{\top} y)^2}{\sum_{1 \leq j \leq n, L_{ji} \neq 0} x_j}\right\} + y^{\top} N y + \sum_{i = 1}^{n} \delta_i \frac{y_i^2}{x_i} + g^{\top} y + h^{\top} x\\
& + \mathcal I_{F_{1}}(y) + \mathcal I_{F_{2}}(x, y). 
\end{align*}

\begin{proposition}
For any $x \in [0,1]^{n}$, $S$ is the set of indices of nonzero x entries, and $S^{C}$ denotes its complement. Suppose 
\begin{align*}
\scriptsize
& I^{>} = \{i:\sum_{1\leq j \leq n, L_{ji} \neq 0} x_j > 1, i = 1, \dots, K\},  \\
&  I^{=} = \{i:\sum_{1\leq j \leq n, L_{ji} \neq 0} x_j = 1, i = 1, \dots, K\},  \\
&  I^{<} = \{i:0< \sum_{1\leq j \leq n, L_{ji} \neq 0} x_j < 1, i = 1, \dots, K\}, \\
&  I^{0} = \{i:\sum_{1\leq j \leq n, L_{ji} \neq 0} x_j = 0, i = 1, \dots, K\}.
\normalsize
\end{align*}
Let $y \in \R^{n}$ ,  $y_i = 0$ for $i \in S^{C}$ , and $(y_{S},\lambda,\mu)$ be the primal-dual optimal solution of the following QP:
\begin{subequations}\label{prop2:QP}
 \begin{align}
\min_{\tilde y \in \R^{|S|}} \quad & \tilde y^{\top} \left(\sum_{i \in I^{>} \bigcup I^{=}} (L_i)_{S} (L_i)_{S}^{\top} + \sum_{i \in I^{<}} \frac{1}{\sum_{1 \leq l \leq n, L_{li} \neq 0} x_l} (L_i)_{S} (L_i)_{S}^{\top} + \text{diag}\left(\left\{\frac{\delta_i}{x_i} \right\}_{i \in S} \right) + N_{S}\right) \tilde y + g_{S}^{\top} \tilde y \\
\text{s.t.} \quad & A_{S} \tilde y \leq b \\
                & C_{S} \tilde y \leq D x,
\end{align}  
\end{subequations}

Then, $t \in \partial \bar{f}_{\text{persp+ro}}(x)$ if and only if there exist $\phi \in \R^{|I^{=}| + |I^{0}|}$ , $\psi \in \R^{|I^{0}|}$, $\Psi \in \R^{|S^{C}|}$, and $\Phi \in \R^{|S^{C}|}$ such that 
\small{
\begin{align}
 0 =&\sum_{i \in I^{>} \bigcup I^{=}} 2L_{ji}(L_i^{\top} y)  +  \sum_{i \in I^{<}} 2\frac{1}{\sum_{1 \leq j \leq n, L_{ji} \neq 0} x_j}L_{ji}(L_i^{\top} y) + \sum_{i \in I^{0}}  L_{ji} \psi_i \nonumber\\
 &+ \delta_j \Psi_j 
 + 2 N_{i}^{\top} y + g_j + \lambda^{\top} A_j + \mu^{\top} C_j = 0, \quad    \forall j \in S^{C} \label{prop2:const1}\\
 t_j = &\sum_{i \in I^{=}, L_{ji} \neq 0} \phi_i - \sum_{i \in I^{<}, L_{ji} \neq 0} \frac{(L_{i}^{\top}y)^2}{(\sum_{1\leq l \leq n, L_{li} \neq 0} x_l)^2} - \delta_j \frac{y_j^2}{x_j^2} - \tilde{\mu}^{\top} D_j + h_j,\quad  \forall j \in S \label{prop2:const2}\\
 t_{j}  = &\sum_{i \in I^{=},L_{ji} \neq 0} \phi_{i} - \sum_{i \in I^{<}, L_{ji} \neq 0} \frac{(L_{i}^{\top}y)^2}{(\sum_{1\leq l \leq n, L_{li} \neq 0} x_l)^2} + \sum_{i \in I^{0}, L_{ji} \neq 0} \phi_{i} + \delta_j \Phi_j - \mu^{\top} D_j + h_j, \quad \forall j \in S^{C} \label{prop2:const3}\\
 \phi_i \geq & -(L_{i}^{\top} y)^2, \quad   \forall i \in I^{=} \label{prop2:const4}\\
 \phi_i \leq  &- \frac{1}{4} \psi_i^2, \quad \forall i \in I^{0} \label{prop2:const5}\\
 \Phi_i \leq &- \frac{1}{4} \Psi_i^2, \quad  \forall 
i \in S^{C} 
\label{prop2:const6}.
\end{align}
}
\end{proposition}

\begin{proof}
Note that $f_{\text{persp+ro}}(x,y) \geq \sum_{i = 1}^{n} \delta_i \frac{y_i^2}{x_i} + y^{\top}(LL^{\top}) y + y^{\top} N y + g^{\top} y + h^{\top} x$, and the right-hand side can be considered as a perspective reformulation. Since we have shown in the proof of Proposition ~\ref{prop:1} that the level set boundedness condition in Theorem ~\ref{thm:1} holds for $f_{\text{persp}}$, it also holds for $f_{\text{persp+ro}}$. Then, the subdifferential of the marginal function $ \bar{f}_{\text{persp+ro}}(x)$ can be obtained by solving for the first-order condition in Theorem ~\ref{thm:1}. By the same argument as in the proof of Proposition~\ref{prop:1}, we can compute $\partial f_{\text{persp+ro}}(x, y)$ as the sum of the subdifferentials of the components of $f_{\text{persp+ro}}$.  The only difference between $f_{\text{persp+ro}}$ and $f_{\text{persp}}$ is the appearance of the term $\sum_{i=1}^{K} \max \left\{(L_{i}^{\top} y)^2, \frac{(L_{i}^{\top} y)^{2}}{\sum_{1 \leq j \leq n, L_{ji} \neq 0} x_j} \right\}$, by Lemma ~\ref{lem:2} and the calculus of subdifferentials, its subdifferential set equals
\begin{align*}
& \sum_{i \in I^{>}} (0, 2 (L_i^{\top} y)L_i)  + \\
& \sum_{i \in I^{=}} \conv \{(0, 2 (L_{i}^{\top} y)L_{i}), (\underbrace{-(L_{i}^{\top} y)^2, \dots, -(L_{i}^{\top} y)^2}_{\text{indices where} \; L_{li} \neq 0},\underbrace{0,\dots, 0}_{\text{indices where} \; L_{li} = 0}, 2(L_{i}^{\top}y)L_{i})\} + \\
& \sum_{i \in I^{<}} (\underbrace{-\frac{(L_{i}^{\top}y)^2}{(\sum_{1 \leq j \leq n, L_{ji} \neq 0} x_j)^2}, \dots, -\frac{(L_{i}^{\top}y)^2}{(\sum_{1 \leq j \leq n, L_{ji} \neq 0} x_j)^2}}_{\text{indices where} \; L_{li} \neq 0}, \underbrace{0, \dots, 0}_{\text{indices where} \; L_{li} = 0}, 2 \frac{1}{\sum_{1 \leq j \leq n, L_{ji} \neq 0} x_j}  (L_i^{\top} y)L_i) + \\
& \sum_{i \in I^{0}} \{(\underbrace{\phi_i,\dots,\phi_i }_{\text{indices where} \; L_{li} \neq 0}, \underbrace{0,\dots,0}_{\text{indices where} \; L_{li} = 0}, \psi_i L_i):\phi_i \leq - \frac{1}{4} \psi_{i}^2\}.
\end{align*}
The subdifferential set of $\max \left\{(L_{i}^{\top} y)^2, \frac{(L_{i}^{\top}y)^2}{\sum_{1 \leq j \leq n, L_{ji} \neq 0} x_{j}}\right\}$ for some $i \in I^{=}$ can be written more compactly as $\{(\underbrace{\phi_i,\dots,\phi_i}_{\text{indices where} \; L_{li} \neq 0}, \underbrace{0,\dots,0}_{\text{indices where} \; L_{li} = 0}, 2 (L_i^{\top}y) L_i):0 \geq \phi_{i} \geq -(L_i^{\top}y)^2\}$.  The subdifferentials of other terms in $f_{\text{persp+ro}}$ are the same as $f_{\text{persp}}$. Now, we state again that $t \in \partial \bar{f}_{\text{persp+ro}}(x)$  if and only if there exists $y \in \R^{n}$ such that
\begin{equation}\label{pro2:eqn1}
(t, 0) \in \partial f_{\text{persp+ro}}(x, y).
\end{equation}

 Similarly, for the subdifferential set $\partial f_{\text{persp+ro}}(x,y)$ to exist, we must have $y_i = 0$ for $i \in S^{C}$ since otherwise, $\delta_i \frac{y_i^2}{x_i} = +\infty$, and $(x, y)$ must satisfy constraints \eqref{MIQP:const1}--\eqref{MIQP:const2}. Recall that the subdifferential sets of the indicator functions $\mathcal{I}_{\mathit{F}_1}$ and  $\mathcal{I}_{\mathit{F}_{2}}$  are 
 \begin{align*}
& \mathcal N_{\mathit{F}_{1}} (y)  = \{(0, \tilde{\lambda}^{\top} A) \; | \; \tilde{\lambda} \geq 0, \tilde{\lambda}_{i} = 0, \; \forall i \text{ such that} \; \langle a_i, y \rangle < b_i\},    \\
& \mathcal N_{\mathit{F}_{2}} (x,y) = \{ (-\tilde{\mu}^{\top} D, \tilde{\mu}^{\top} C) \; | \; \tilde{\mu} \geq 0, \tilde{\mu}_{i} = 0, \; \forall i \text{ such that} \; \langle c_i, y \rangle  < \langle d_i, x \rangle \}.
\end{align*}
 Again, we uncover equation \eqref{pro2:eqn1} entry by entry in the order of (i) $y_j$ for $j \in S$, (ii) $y_j$ for $j \in S^{C}$, (iii) $x_j$ for $j \in S$, and (iv) $x_j$ for $j \in S^{C}$.  The first set of equations is for $j \in S$:
 \small{
\begin{align}\label{prop2:eqn1-1}
& \sum_{i \in I^{>} \bigcup I^{=}} 2L_{ji}(L_i^{\top} y) + \sum_{i \in I^{<}} 2\frac{1}{\sum_{1 \leq j \leq n, L_{ji} \neq 0} x_j}L_{ji}(L_i^{\top} y) + 2\delta_j \frac{y_j}{x_j} + 2 N_{i}^{\top} y + g_j + \tilde{\lambda}^{\top} A_j + \tilde{\mu}^{\top} C_j = 0.
\end{align}}
The second set of equations is for $j \in S^{C}$:
\small{
\begin{align}\label{prop2:eqn1-2}
& \sum_{i \in I^{>} \bigcup I^{=}} 2L_{ji}(L_i^{\top} y) + \sum_{i \in I^{<}} 2\frac{1}{\sum_{1 \leq j \leq n, L_{ji} \neq 0} x_j}L_{ji}(L_i^{\top} y) + \sum_{i \in I^{0}}  L_{ji} \psi_i + \delta_j \Psi_j + 2 N_{i}^{\top} y + g_j + \tilde{\lambda}^{\top} A_j + \tilde{\mu}^{\top} C_j = 0,
\end{align}}
which is exactly \eqref{prop2:const1}. The third set of equations is for $j \in S$:
\small{
\begin{align}\label{prop2:eqn1-3}
 & t_j = \sum_{i \in I^{=}, L_{ji} \neq 0} \phi_i - \sum_{i \in I^{<}, L_{ji} \neq 0} \frac{(L_{i}^{\top}y)^2}{(\sum_{1\leq l \leq n, L_{li} \neq 0} x_l)^2} - \delta_j \frac{y_j^2}{x_j^2} - \tilde{\mu}^{\top} D_j + h_j,  
\end{align}}
which is nothing but \eqref{prop2:const2}. The last set of equations is for $j \in S^{C}$:
\begin{align}\label{prop2:eqn1-4}
 &t_{j} = \sum_{i \in I^{=}, L_{ji} \neq 0} \phi_{i} - \sum_{i \in I^{<}, L_{ji} \neq 0} \frac{(L_{i}^{\top}y)^2}{(\sum_{1\leq l \leq n, L_{li} \neq 0} x_l)^2} + \sum_{i \in I^{0}, L_{ji} \neq 0} \phi_{i} + \delta_j \Phi_j - \tilde{\mu}^{\top} D_j + h_j,
\end{align}
which is equivalent to \eqref{prop2:const3}. In equations \eqref{prop2:eqn1-2}--\eqref{prop2:eqn1-4}, $\psi$, $\Psi$, $\phi$, and $\Phi$ satisfy the relations \eqref{prop2:const4}--\eqref{prop2:const6}. 
Note that equations \eqref{prop2:eqn1-1} where $\tilde{\lambda}$ and $\tilde{\mu}$ satisfy the complementarity condition together with the conditions that the given $(x,y)$ satisfies the constraints \eqref{MIQP:const1}--\eqref{MIQP:const2}  and $y_i = 0$ for all $i \in S^{C}$ are equivalent to saying that the triple $(y_S, \tilde{\lambda}, \tilde{\mu})$ satisfies the KKT condition for the QP \eqref{prop2:QP}.
\end{proof}

Now, a question to ask is how to decompose $Q$ as $\sum_{i = 1}^{K} L_i L_i^{\top} + \text{diag}(\delta) + N$? \citet{atamturk2019rank} propose a semidefinite programming (SDP) approach for decomposing $Q$ into the sum of low-dimensional rank-one quadratic functions and one-dimensional quadratic functions for the sake of maximizing the relaxation lower bound; however, the computational cost of solving the SDP problem using interior-point methods could be prohibitive. A much cheaper option is the Cholesky decomposition, and we refer interested readers to  \cite{golub2013matrix} for a complete discussion on variants of the Cholesky decomposition and their numerical stability. We can divide the task into (i) first extract a diagonal term $\text{diag}(\delta)$ from $Q$, and then (ii) do a Cholesky decomposition on the remaining positive semidefinite matrix (the Cholesky decomposition is applicable for both positive definite and positive semidefinite matrices). 

The subgradient cuts of $\bar f_{\text{persp+ro}}$ have a much simpler form at binary points.  In what follows, we will present an algorithm for computing cutting planes in this case. 

\begin{algorithm}[bht]
\caption{A cut-generating procedure based on the perspective reformulation and rank-one inequalities ($\text{CUT}_{\text{persp+ro}}$)} \label{alg:2}
\begin{algorithmic}[1]
\STATE \textbf{INPUT} $R, \delta,L,  g,h, (A,b), (C,D), x$
\STATE Set $S$ to be the set of indices such that $x_i > 0$
\STATE Set $I^{0}$ to be the set of indices such that $\sum_{1 \leq l \leq n, L_{li} \neq 0} x_l = 0$
\STATE Set $y = 0$
\STATE Set $(y_{S}, \lambda, \mu)$ to be the optimal primal-dual solution of QP 
\eqref{prop1:QP}
\STATE Set $t = 0$
\FOR{$i \in S$}
\STATE Set $t_i = - \delta_i \frac{y_i^2}{x_i^2} - \mu^{\top} D_{i} + h_i$ 
\ENDFOR
\FOR{$i \in I^{0}$}
\STATE Set $n_i = |\{j \; : \; L_{ji} \neq 0 \} |$
\ENDFOR
\FOR{$i \in S^{C}$}
\STATE Set $r_i = -2\sum_{j \in S}^{} R_{ij} y_j - g_i - \lambda^{\top} A_i - \mu^{\top} C_i$
\ENDFOR
\STATE Solve the system \small{$\left(\text{diag}\left(\left\{\delta_i\right\}_{i \in S^{C}}\right) + \sum_{i \in I^{0}} \frac{1}{n_i}(L_i)_{S^{C}} (L_i)_{S^{C}}^{\top} \right) \beta = 2r$}
\normalsize
\FOR{$i \in S^{C}$}
\STATE Set $\Psi_i = \frac{1}{2} \beta_i$
\ENDFOR
\FOR{$i \in I^{0}$}
\STATE Set $\psi_i = \frac{1}{2 n_i} \sum_{j \in S^{C}} L_{ji} \beta_j$
\ENDFOR
\FOR{$i \in S^{C}$}
\STATE Set $t_i = - \frac{\delta_i^2}{4} \Psi_i^2 - \frac{1}{4} \sum_{j \in I^{0}, L_{ij} \neq 0} \psi_j^2 - \mu^{\top} D_{i} + h_i$
\ENDFOR
\STATE \textbf{return} $t$
\end{algorithmic}
\end{algorithm}

\begin{proposition}\label{prop:2}
Given $ x \in X \subset \{0,1\}^{n}$, suppose we decompose $Q$ as $Q = \sum_{i = 1}^{K} L_i L_i^{\top} + \text{diag}(\delta) + N$ such that $\delta \geq 0$, $N \succcurlyeq 0$ and $R = \sum_{i = 1}^{K} L_i L_i^{\top} + N$. Let $t$ be the vector returned by Algorithm \ref{alg:2}, then $t \in \partial \bar f_{\text{persp+ro}}(x)$ and $\tilde \eta \geq \langle t, \tilde x -  x \rangle + \bar f_{\text{persp+ro}}(x)$ is a valid cut for $\mathcal Z$.
\end{proposition}

\begin{proof}
Since $x \in X$ is binary, the index set $I^{<}$ does not exist. Thus, QP \eqref{prop2:QP} reduces to:

\begin{subequations}
\allowdisplaybreaks
\begin{align}
\min_{\tilde y \in \R^{|S|}} \quad & \tilde y^{\top} \left(\sum_{i \in I^{>} \bigcup I^{+}} (L_{i})_{S} (L_{i})_{S}^{\top} + \text{diag}\left(\left\{\frac{\delta_i}{x_i} \right\}_{i \in S}\right)+N_{S}\right) \tilde y + g_{S}^{\top} \tilde y \\
\text{s.t.} \quad & A_{S} \tilde y \leq b \\
                & C_{S} \tilde y \leq D x.
\end{align}
\end{subequations}
Since each $L_i$ for $i \in I^{0}$ is supported only on $S^{C}$, the Hessian of the objective equals $$\sum_{i \in I^{>} \bigcup I^{=} \bigcup I^{0}} (L_i)_{S} (L_{i})_{S}^{\top} + \text{diag}\left(\left\{\frac{\delta_i}{x_i} \right\}_{i \in S} \right) + N_{S},$$ which by definition is also equal to $\text{diag}\left(\left\{\frac{\delta_i}{x_i} \right\}_{i \in S} \right) + R_{S}$. If we take $\phi_i = 0$ for $i \in I^{=}$ in equations \eqref{prop2:const2}--\eqref{prop2:const4}, then equations \eqref{prop2:const1}--\eqref{prop2:const6} further simplify to 
\begin{align}
    & t_i = - \delta_i \frac{y_i^2}{x_i^2} - \mu^{\top} D_{i} + h_i, \quad \forall i \in S \label{prop2:cond7}\\
    & \delta_i \Psi_i + \sum_{j \in I^{0}, L_{ij} \neq 0} L_{ij} \psi_j = -(2 N_{i}^{\top} y + \sum_{j \in I^{>} \bigcup I^{=}} 2 L_{ij} (L_{j}^{\top} y) + g_i + \lambda^{\top} A_{i} + \mu^{\top} C_{i}), \quad \forall i \in S^{C} \label{prop2:cond8}\\
    & t_i = \delta_i \Phi_i + \sum_{j \in I^{0}, L_{ij} \neq 0} \phi_j  - \mu^{\top} D_{i} + h_i, \quad \forall i \in S^{C}\label{prop2:cond9} \\
& \phi_i \leq - \frac{1}{4} \psi_i^2, \quad \quad \quad  \forall i \in I^{0} \label{prop2:cond10}\\
& \Phi_i \leq - \frac{1}{4} \Psi_i^2, \quad \quad \quad \forall 
i \in S^{C}. \label{prop2:cond11}
\end{align}
We can substitute \eqref{prop2:cond10}--\eqref{prop2:cond11} into \eqref{prop2:cond9}, to obtain
\begin{align}\label{prop2:cond9+}
 t_i \leq -\frac{\delta_i}{4} \Psi_i^2  - \frac{1}{4} \sum_{j \in I^{0}, L_{ij} \neq 0} \psi_j^2  - \mu^{\top} D_{i} + h_i, \quad \forall i \in S^{C}.    
\end{align}

The free variables are $\Psi_i$ for $i \in S^{C}$, $\psi_i$ for $i \in I^{0}$, and $t_i$ for $i \in S^{C}$. Note that a cut $\tilde \eta \geq \langle \bar t, \tilde x- x \rangle + \bar f_{\text{persp+ro}}(x)$ `dominates' another cut $\tilde \eta \geq \langle \tilde t, \bar x - x \rangle + \bar f_{\text{persp+ro}}(x)$ if $\bar t_i \geq \tilde t_i, \ \forall i \in S^{C}$ (because the latter is implied by the former together with the bounds $-x_i \leq 0$), thus one strategy is to select $\Psi, \psi, t$ such that the summation $\sum_{i \in S^{C}} t_i$ is maximized, and this guarantees that the cut is non-dominated. Let us denote the right-hand side of \eqref{prop2:cond8} as $r_i$ for all $i \in S^{C}$. Again, since each $L_i$ for $i \in I^{0}$ is only supported on $S^{C}$, the right-hand side of \eqref{prop2:cond8} equals $-(2R_{i}^{\top} y + h_i + \lambda^{\top} A_i + \mu^{\top} C_i)$. The QP problem for maximizing $\sum_{i \in S^{C}} t_i$ is as follows:
\begin{subequations}\label{prop2:subprob}
\begin{align}
\min \quad & \sum_{i \in S^{C}} \delta_i \Psi_i^2 + \sum_{i \in I^{0}} \sum_{j \in S^{C}, L_{ji} \neq 0}^{} \psi_i^2 \\
\text{s.t.} \quad & \delta_i \Psi_i + \sum_{j \in I^{0}} L_{ij} \psi_j = r_i \quad [\beta_i] \quad \forall i \in S^{C} .
\end{align}
\end{subequations}
We associate each constraint with Lagrangian multiplier $\beta_i$, and then the KKT conditions of \eqref{prop2:subprob} are:
\begin{align}
& \Psi_i = \frac{1}{2} \beta_i, \quad \forall i \in S^{C} \label{kkt:con1}\\
& \psi_i =  \frac{\sum_{j \in S^{C}} L_{ji} \beta_j}{2 n_i}, \quad \quad \quad  \forall i \in I^{0} \label{kkt:con2}\\
& \delta_i \Psi_i + \sum_{j \in I^{0}} L_{ij} \psi_j = r_i, \quad \quad \forall i \in S^{C} \label{kkt:con3}.
\end{align}
By substituting equations \eqref{kkt:con1} and \eqref{kkt:con2} into \eqref{kkt:con3}, we can solve for the multiplier $\beta$ via 
\begin{equation}\label{prop2:linearsys}
\left(\text{diag}\left(\left\{\delta_i\right\}_{i \in S^{C}}\right) + \sum_{i \in I^{0}} \frac{1}{n_i}(L_i)_{S^{C}} (L_i)_{S^{C}}^{\top} \right) \beta = 2r. 
\end{equation}
Since the coefficient matrix is a diagonal matrix plus a rank $|I^{0}|$ term, its inverse can be computed efficiently as the sum of a diagonal matrix and rank $|I^{0}|$ matrix using the Woodbury matrix identity. 

Finally, we can use \eqref{kkt:con1}, \eqref{kkt:con2}, and \eqref{prop2:cond9+} to obtain $t_i$ for $i \in S^{C}$.
\end{proof}

\begin{remark}\label{rem:4}
If we decompose $Q$ as $Q = \sum_{i = 1}^{K} L_i L_i^{\top} + \text{diag}(\delta) + N$ such that $\delta \geq 0, N \succcurlyeq 0$, and $R = \sum_{i = 1}^{K} L_i L_i^{\top} + N$, then at any binary point $x \in X$, $\partial \bar{f}_{\text{persp}}(x) \subsetneqq \partial \bar{f}_{\text{persp+ro}}(x)$.
\end{remark}

This strict inclusion relation can be seen from the fact that if we set $\psi_i = 0$ for all $i \in I^{=} \bigcup I^{0}$ and $\phi_i = 0$ for all $i \in I^{0}$, then equations \eqref{prop2:cond8} and \eqref{prop2:cond9} are identical to equations \eqref{prop1:statement2} and \eqref{prop1:statement3}, and thus any $t$ feasible for the system \eqref{prop1:statement1}--\eqref{prop1:statement3} is also feasible for the system \eqref{prop2:cond7}--\eqref{prop2:cond11}.

\begin{remark}\label{rem:5}
If we decompose $Q$ as $Q = \sum_{i = 1}^{K} L_i L_i^{\top} + \text{diag}(\delta) + N$ such that $\delta \geq 0, N \succcurlyeq 0$, and $R = \sum_{i = 1}^{K} L_i L_i^{\top} + N$, then at any binary point $x \in X$, $\bar{f}_{\text{persp}}(x) = \bar{f}_{\text{persp+ro}}(x)$. For any two $t$ and $\tilde t$ returned by Algorithm ~\ref{alg:3} and Algorithm ~\ref{alg:2} respectively, we always have $\sum_{i \in S^{C}} \tilde t_i \geq \sum_{i \in S^{C}} t_i$.
\end{remark}

For any $i \in I^{>} \bigcup I^{=}$, the term $\max\left\{(L_i^{\top}y)^2, \frac{(L_i^{\top} y)^2}{\sum_{1 \leq j \leq n, L_{ji} \neq 0} x_j}\right\}$ equals $(L_i^{\top} y)^2$. For any $i \in I^{0}$, the fact $\sum_{1 \leq j \leq n, L_{ji} \neq 0} x_j = 0$ implies that $L_{ji} \neq 0$ only if $x_j = 0$. If there exists some $j \in S^{C}$ such that $y_j \neq 0$, then both $f_{\text{persp}}(x, y)$ and $f_{\text{persp+ro}}(x, y)$ will be $+\infty$ as they have the term $\delta_j \frac{y_j^2}{x_j}$. On the other hand, if $y_j = 0$ for all $j \in S^{C}$, then the term $\max\left\{(L_i^{\top}y)^2, \frac{(L_i^{\top} y)^2}{\sum_{1 \leq j \leq n, L_{ji} \neq 0} x_j}\right\}$ vanishes for any $i \in I^{0}$, and $y^{\top} (\sum_{i \in I^{>} \bigcup I^{=}} L_i L_i^{\top} + \text{diag}(\delta) + N) y = y^{\top} (\text{diag}(\delta) + R) y = y^{\top} Q y$. Overall, we have $f_{\text{persp}}(x, y) = f_{\text{persp+ro}}(x, y)$ for any $y \in \R^{n}$, and thus $\bar{f}_{\text{persp}}(x) = \bar{f}_{\text{persp+ro}}(x)$. The second part holds since by Remark ~\ref{rem:4}, any $t \in \partial \bar{f}_{\text{persp}}(x)$ is also a subgradient of $\bar{f}_{\text{persp+ro}}(x)$, and Algorithm ~\ref{alg:2} outputs the subgradient vector $\tilde t$ of $\bar{f}_{\text{persp+ro}}(x)$ such that $\sum_{i \in S^{C}} \tilde t_i$ is maximized. Remark ~\ref{rem:5} confirms the intuition that the tighter convex envelope $f_{\text{persp+ro}}$ leads to stronger cutting planes of $\mathcal{Z}$.

The computational cost of Algorithm ~\ref{alg:2} equals the cost of solving QP \eqref{prop1:QP} plus $\mathcal{O}(|I^{0}|^2n + |I^{0}|^3)$ floating-point operations (the cost of solving the linear system \eqref{prop2:linearsys}). Again, the procedure can be very efficient if we are at a binary point $x$ satisfying a strong sparsity constraint. On the other hand, we will only generate cuts at binary points when we utilize Algorithm ~\ref{alg:2} as the cut-generating procedure in an outer approximation algorithm. 

In the next section, we will demonstrate the effectiveness of the cutting planes in Proposition ~\ref{prop:1} and Proposition \ref{prop:2} on sparse portfolio selection problems with minimum and maximum investment constraints. We will benchmark the OA algorithms with Algorithm ~\ref{alg:3} and Algorithm ~\ref{alg:2} as the cut-generation procedure against the state-of-the-art methods in the literature. 

\section{Computational experiments}\label{sec:3}
In this section, we demonstrate the practical performance of the cutting planes in Proposition ~\ref{prop:1} and Proposition ~\ref{prop:2} through a computational study on the sparse portfolio selection problem. We solve the corresponding MIQP by an OA algorithm in which we iteratively refine the polyhedral approximation for $\mathcal Z$ using cutting planes in Proposition ~\ref{prop:1} and Proposition ~\ref{prop:2}. An outline of the OA algorithm is given in~Algorithm \ref{alg:4}. All the experiments are conducted on a laptop with a $12$th Generation Intel(R) Core(TM) i7-12700H CPU and 16 GB RAM. The quadratic programming subproblems are modeled by CVX 2.2 and solved using MOSEK 10.1.13, and the mixed-integer programming (MIP) master problem is solved by Gurobi 10.0.1. We adapt the default settings of Gurobi, except that the number of threads in use is forced to be one. The MIP solver terminates when a solution of relative optimality gap $< 0.01\%$ is found. All the algorithms and mathematical optimization models are implemented in Python.

\subsection{Outer approximation}
The OA algorithm is formalized in Algorithm ~\ref{alg:4}. 
\begin{theorem}\label{thm:2}
Algorithm ~\ref{alg:4} solves MIQP \eqref{eqn:MIQP} exactly in a finite number of iterations.
\end{theorem}
\begin{proof}
Let us denote the optimal value of MIQP \eqref{eqn:MIQP} as \textbf{OPT}. The finite convergence of Algorithm ~\ref{alg:4} follows from the fact that each binary solution is visited at most once, and only a finite number of cutting planes are added. Suppose Algorithm ~\ref{alg:4} returns a solution $(x^{T}, \eta^{T})$, then because $\mathcal{Z}_{t}$ contains $\mathcal{Z}$ at iteration $t$ (the cutting planes are valid for $\mathcal{Z})$, we have $\eta^{T} \leq \text{OPT}$. By the stopping criteria, we have $\eta^{T} \geq \bar{f}_{\text{persp}}(x^{T})$. Remark ~\ref{rem:6} tells that $\bar{f}_{\text{persp}}(x^{T})$ is the optimal value of MIQP when $x$ is fixed to $x^{T}$, and hence upper bounds \textbf{OPT}. Putting these together, we have $\eta^{T} = \text{OPT}$.
\end{proof}

Instead of rebuilding a branch-and-bound tree each time a violated cut is added, we adapt the single-tree strategy through the \textit{lazy constraints}. This means that the cutting planes are added to the MIP solver when an integer solution is found, and it violates some cutting plane. The solver terminates when no more violated cutting planes are added, and the current integer solution is optimal. 

\subsubsection{Incumbent solutions:} High-quality incumbent solutions help prune the branch-and-bound nodes and save efforts on exploring suboptimal regions. In the literature, heuristics exist under some particular constraints of \eqref{eqn:MIQP}, i.e., \cite{xie2020scalable, Bertsimas2016, bertsimas2022scalable}; however, there does not exist a heuristic that can generate provably near-optimal solutions for MIQP in its general form. Another strategy is providing the visited binary points $(x^t, \bar{f}_{\text{persp}}(x^t))$ as incumbent solutions to the MIP solver (we have established in Theorem ~\ref{thm:2} that $\bar{f}_{\text{persp}}(x^{t})$ upper bounds the optimal solution). Implementing this strategy takes no additional cost to the OA procedure as the computation of $\bar{f}_{\text{persp}}(x^t)$ is included in the cut generation, and we can set incumbent solutions through a \textbf{user callback} function for commercial solvers such as CPLEX and Gurobi. 

\subsubsection{Copy variables:} We add a copy of the constraints $A y \leq b$ and $C y \leq D x$ to the initial polyhedral approximation $\mathcal{Z}_{1}$, i.e., $Z_{1} = \{(x,y,\eta):A y \leq b, C y \leq D x, x \in \conv(X)\}$. Doing so guarantees that the quadratic programming subproblem \eqref{prop1:QP} is always feasible when we generate cutting planes. 

\subsubsection{Improve the lower bound:} Note that the perspective relaxation provides a valid lower bound on the optimal value of the MIQP \eqref{eqn:MIQP}.  We denote the optimal value of the perspective relaxation as $\eta_{\text{persp}}$; then we can add the inequality $\eta \geq \eta_{\text{persp}}$ to the MIP solver. In some scenarios, we must obtain a high-quality solution before the MIP solver terminates. A near-optimal lower bound, such as $\eta_{\text{persp}}$, will help close the optimality gap of the best integer solution found. A stronger relaxation lower bound can be obtained via SDP relaxations, i.e., \cite{dong2015regularization, zheng2014improving, atamturk2019rank, HGA:2x2}, and we do not exploit these options due to the high computational cost of solving SDPs. 

\subsubsection{Decomposition:} In terms of the diagonal decomposition, i.e., $Q = R + \text{diag}(d)$ such that $R \succcurlyeq 0$ and $d \geq 0$, several approaches exist. For example, \citet{Frangioni2007} propose solving a simple SDP while \citet{zheng2014improving}
propose solving a larger-sized SDP, which seeks the diagonal decomposition such that the lower bound of the resulting perspective relaxation is maximized. \citet{zheng2014improving} also test the convex combination of those two. In our experiments, the decomposition method in \cite{Frangioni2007} attains a better solution time for almost all the instances and solution methods. Thus, we only present the results using the method in \cite{Frangioni2007} to extract a diagonal term. The rank-one decomposition is obtained by doing Cholesky decomposition on the remainder term $R$, i.e., $Q = \sum_{i = 1}^{K} L_i L_i^{\top} + \text{diag}(\delta)$ such that $\sum_{i = 1}^{K} L_i L_i^{\top} = R$.

\begin{algorithm}[H]
\caption{An outer approximation algorithm (OA)} \label{alg:4}
\begin{algorithmic}[1]
\STATE Set $t = 1$
\STATE Set $\mathcal Z_{1} = \{(x,y,\eta):Ay \leq b, Cy \leq Dx, x \in \conv(X), \eta \geq \eta_{\text{persp}} \}$ 
\WHILE{$\eta^t < \bar{f}_{\text{persp}}(x^t)$}
\STATE Compute the solution $(x^t, \eta^t)$ of
$$\min_{x \in X, \eta} \quad \eta \quad \text{s.t.} \quad (x, y, \eta) \in \mathcal{Z}_{t} \; \text{for some} \; y$$
\STATE Add a cut $t$ using Algorithm ~\ref{alg:3} or \ref{alg:2}
$$
\mathcal{Z}_{t} = \mathcal{Z}_{t} \bigcap \{(x,y,  \eta) : \eta \geq \langle t, x - x^t \rangle + \bar{f}_{\text{persp}}(x^t) \}
$$
\STATE Set $t = t + 1$
\ENDWHILE
\end{algorithmic}
\end{algorithm}
\subsection{Portfolio selection problems} Given a universe of $n$ risky assets, we denote the mean return vector as $\mu$ and the covariance matrix as $Q$. The mean-variance portfolio selection problem with sparsity, minimum investment, maximum investment, and minimum return constraints can be formulated as follows:
\begin{subequations}\label{eqn:portfolio}
\begin{align}
\min_{x, y} \quad & y^{\top} Q y \\
\text{s.t.} \quad & \sum_{i = 1}^{n} y_i = 1 \\
                  & \mu^{\top} y \geq \rho  \\
                  & \alpha_i x_i \leq y_i \leq u_i x_i, \quad \forall i \in [n]  \\
                  & \sum_{i = 1}^{n} x_i \leq k \\
                  & x \in \{0,1\}^{n}. 
\end{align}
\end{subequations}

The random instances we use are generated the same way as the diagonal-dominant instances in \cite{Frangioni2007}. The expected return ($\mu_i$), minimum investment ($\alpha_i$), and maximum investment ($u_i$) are drawn from uniform distributions with ranges [0.002, 0.01], [0.075, 0.125], and [0.375, 0.425], for each $i=1,\dots,n$,  respectively. The minimum return value $\rho$ is also from a uniform distribution over the interval [0.002, 0.01]. The off-diagonal entries of $Q$ are drawn from a discrete uniform distribution between 1 and 10, and the diagonal entries are drawn from a discrete uniform distribution between $10n$ and $20n$. Since $\alpha_i \geq 0.075$, the number of nonzeros in the vector $y$ is at most 13. We test with cardinality constraints when $k = 6, 8, 10$ and the case without a cardinality constraint, which we denote by `nc'. For $n = 300, 400$, we reuse these publicly available instances at \url{https://commalab.di.unipi.it/datasets/MV/} for benchmarking, and for larger problem sizes, which are not available in the public data set, we generate new random instances following the same scheme.

We compare the following solution methods.

\begin{itemize}
\itemsep 1em
  \item  MISOCP: the mixed-integer second-order conic reformulation of the perspective reformulation of \eqref{eqn:portfolio} solved by Gurobi. 

\item OA-BC: the OA method proposed by \citet{bertsimas2022scalable}.

\item OA-persp: Algorithm ~\ref{alg:4} with Algorithm ~\ref{alg:3} as the cut generating subroutine.
\item  OA-persp+ro: Algorithm ~\ref{alg:4} with Algorithm ~\ref{alg:2} as the cut generating subroutine.
\end{itemize}

Note that the original OA algorithm in \cite{bertsimas2022scalable} is implemented in Julia, and here, we re-implement it in Python for a fair comparison with the other methods implemented in Python. Authors in \cite{bertsimas2022scalable} also supply additional cuts at the root node to strengthen the original polyhedral approximation of $\mathcal Z$. We also test the OA algorithm with additional cuts generated by an in-out bundle method at the root node; however, although this implementation improves the solution time in the OA procedure, it takes more time overall because the cost of generating additional cuts is substantial. Another difference in the experiment settings is the choice of MIP solver. In \cite{bertsimas2022scalable, zheng2014improving, Frangioni2007}, the  MIP solver in use is CPLEX, whereas here we use Gurobi.

In Table ~\ref{tab:1}, we present the comparison between MISOCP, OA-BC, OA-persp, and OA-persp+ro, and each entry represents an average of ten instances. A time limit of ten minutes is enforced, and the parenthesis followed by the solution time indicates the number of instances that hit the time limit before reaching an optimal solution. The method that attains the best solution time is shown in bold. We provide instance-wise performance in Appendix \ref{app:suppl} for a more detailed comparison.

By comparing Table~\ref{tab:1} with the tables in \cite{bertsimas2022scalable} (the most recent computational study on the sparse portfolio selection problem), we discover that with the recent developments of Gurobi in mixed-integer conic programming, we are now able to explore more than twice as many nodes as CPLEX version 12.8.0.\ run on a machine with a better configuration within the same amount of time. Also, the comparison between MISOCP and OA-BC is reversed in some instances (MISOCP now attains a better end gap than OA-BC in some instances when they both reach the time limit). Consistent with the findings in \cite{bertsimas2022scalable}, MISOCP performs worst in the case where there is no explicit sparsity constraint but an implicit sparsity constraint because of the minimum investment constraints, and its performance significantly deteriorates (with an end gap $\approx 10\%$) when the problem size scales up to 500. This phenomenon suggests that OA methods are better for the sparse portfolio selection problem in large instances.

 Another observation we make is that OA-BC hits the time limit in most instances.  An even closer look at Table ~\ref{tab:1} and the tables in \cite{bertsimas2022scalable} reveals that for both the original implementation of OA-BC and our re-implementation, it takes roughly the same amount of branch-and-bound nodes and cutting planes to obtain an optimal solution; however, on our computing environment, we are only able to explore less than half of the branch-and-bound nodes within ten minutes.

 As we have shown, the cuts generated by OA-BC and OA-persp are the same, but the cut-generating procedure in the latter one is much more efficient (it solves a QP subproblem of size $k$ as opposed to the QP subproblem of size $n$ in OA-BC). The computational results further confirm that OA-persp significantly outperforms OA-BC. OA-persp solves most instances within the time limit, and its improved efficiency against OA-BC allows the exploration of more branch-and-bound nodes and the generation of more cuts. For many instances, both the smaller ones ($n = 300$) and the larger ones ($n = 500$), OA-persp reduces the solution time by one-half or more (this reduction is underestimated as OA-BC fails to solve a large proportion of instances).

Finally, even though OA-persp is the winning algorithm in most cases, the performances of OA-perp and OA-persp+ro tie closely (with only a few instances in which OA-persp+ro takes approximately $10\%$ extra time than OA-persp). Although OA-persp+ro generates stronger cuts than OA-persp, and accordingly maintains a stronger linear programming relaxation, the additional cost of solving the linear system in line 16 of Algorithm ~\ref{alg:2} may slow down the branch-and-bound process. One trend we can observe from Table ~\ref{tab:1} is that the gap between the solution times of OA-persp and OA-persp+ro diminishes as the problem size scales up.

\begin{landscape}
\tiny
\begin{table}[th]
    \centering
    \caption{Comparison results for portfolio selection problems with minimum/maximum return and investment constraints.}
    \vskip 1mm
    \label{tab:1}\begin{tabular}{cccccccccccccccccccc}
    \hline
     & & \multicolumn{3}{c}{MISOCP} & & \multicolumn{4}{c}{OA-BC}  & & \multicolumn{4}{c}{OA-persp} & & \multicolumn{4}{c}{OA-persp+ro} \\
     \cline{3-5} \cline{7-10} \cline{12-15} \cline{17-20}
     Size & $k$ & Gap(\%) & Time(s) & Node &  & Gap & Time & Node & Cut &  &Gap & Time & Node & Cut & & Gap & Time & Node & Cut \\
300 & 6   &  0.13 & 413.57(3) & 561   &  & 0.12 & 360.43(3) & 6337   & 847  &   & 0.00 & \textbf{31.60}  & 7330   & 945  & & 0.00 & 33.62  & 7193   & 938  \\
    & 8  &   0.16 & 470.77(5) & 571   &  & 0.23 & 435.59(4) & 10760  & 1081 &  &  0.00 & \textbf{75.52}   & 18485    & 1566 & & 0.00& 81.36  & 18191 & 1551 \\
    & 10 &   0.11 & 466.51(4) & 571.  &  & 0.32 & 463.79(6)   & 17337  & 1081 &  & 0.09 & \textbf{250.33(2)} & 43794  & 2144 & & 0.09 & 286.12(2) & 47414  & 2092   \\
    & nc &   1,81 & 600.14(10) & 533   &  & 0.12 & 231.94(2) & 250993 & 239  &  & 0.00 & 139.36 & 278212 & 251  & & 0.00 & \textbf{134.59}  & 247410 & 253  \\
400 & 6  &   0.53 & 519.15(8) & 523  &  & 0.47 & 405.62(6) & 5447   & 750  &  & 0.00 & \textbf{72.24}  & 11766    & 1699& & 0.00 & 86.53 & 11205  & 1700 \\
    & 8  &   0.39 & 508.30(8) & 521  &  & 0.41 & 431.89(6) & 5869   & 780  &  & 0.00 & \textbf{225.35} & 30389 & 2855 & &0.00 & 247.60 & 30171 & 2829 \\
    & 10 &   0.20 & 504.27(8) & 518  &  & 0.26 & 442.80(6) & 13698  & 743    &  & 0.08 & \textbf{253.89(3)} & 38300  & 1621 & & 0.08 & 264.43(3) & 41238  & 1587 \\
    & nc &   1.49 & 540.83(9) & 503   &  & 0.41 & 276.14(4) & 421272 & 104  &  & 0.37 & 250.31(4) & 473113 & 106  & &0.41 & \textbf{250.14(4)} & 478389   & 107  \\
     500 & 6 & 0.73 

& 	531.00(9) &482.57 & &  0.34 	
& 451.93(8) & 5283  &638 &   & 0.00
 & \textbf{111} & 15414 &  1843
 & &0.02	
 & 129.28(1) &16171 & 1807   \\
         & 8 & 9.93 
 & 534.05(9) & 466 & & 0.32 
& 487.36(6) & 6130  &649 &   & 0.07	
& \textbf{167.37(1)} & 26360 & 2227 &   & 0.07 
 & 190.01(1) & 28313  & 2145\\
         & 10 & 9.95 
& 516.10(9) & 529 & & 0.37 
& 516.44(9) & 8534 &648 &  & 0.11 
& 349.92(2) & 54403 &1918  &   & 0.11	
 & \textbf{347.95(2)} & 59129 & 1910\\ 
         & nc & 14.22 
& 515.54(9) & 441 & & 1.12 	
& 391.93(7) & 270996 & 161 &   & 1.05 
& \textbf{352(4)} & 363384 & 184  & & 0.93	
  & 357.82(4) &	380651  & 161\\
      \hline
    \end{tabular}
    
\end{table}
\end{landscape}
\normalsize

\section{Conclusion}\label{sec:6}

In this paper, we consider constrained mixed-integer convex quadratic programs (MIQP) with indicators. The motivation for this work arises from a sequence of convexification results on MIQP and the scalability of outer approximation algorithms. We develop a framework to derive projective cutting planes for MIQP, and as special cases, we derive the cutting planes based on perspective reformulation and perspective reformulation with rank-one inequalities. Our approach has some apparent advantages over existing ones. First, we can handle different constraints (potentially nonlinear constraints) by representing them as indicator functions in the objective as long as we can readily compute the resulting subdifferentials. Second, our result applies to general mixed-integer convex programs (MICP) with indicators. Third, an explicit form of the marginal function is not required to derive cutting planes.  Last, it can be easily adapted to other strong reformulations in the literature. The computational results in \S \ref{sec:6} show that the theory translates into promising algorithmic improvements. 

The theoretical framework is not restricted to the problem formulation (MIQP) we study here. One possible extension is applying the framework to other MICP problems, i.e., when the objective is separable convex or the composite of a convex function and a linear mapping.  It is also interesting for future research to study the cutting planes for problems with nonlinear constraints. A theoretical question to answer is, can we relax to the level set boundedness condition? By answering it, we can obtain cutting planes for a broader class of functions. From a computational perspective, an idea for improving the algorithm is to exploit the sparsity of the Hessian. For example, intuitively, the rank-one inequality \eqref{eqn:ro} is stronger when it involves a small proportion of variables, and we can use the sparse Cholesky decomposition to decompose the Hessian into sparse rank-one quadratics.

\section*{Acknowledgement}
This research is supported, in part, by NSF Grant 2007814.

\bibliographystyle{abbrvnat}
\bibliography{reference}


\appendix
\section{Equivalence of Cutting Planes} \label{app:equiv}
Following \citet{bertsimas2022scalable}, and for ease of exposition, we choose $\text{diag}(\delta) = \frac{1}{\gamma} I$ for some $\gamma >0$ (the proof of  the general $\delta$ case is similar). It's easy to verify that the level set boundedness condition holds for $f_{\text{persp}_{2}}$, and by mimicking the steps of the proof of Proposition ~\ref{prop:1}, we conclude that $t \in \partial \bar f_{\text{persp}_2}(x)$ at some $x$ if and only if
\begin{subequations}
\begin{align}
& \exists y \in \R^{n}, \alpha \in \R^{k_1}, \lambda \in \R^{m_1}, \mu \in \R^{m_2}\label{sparsereg:const1}\\
& y_i = 0 \quad \forall i \in S^{C} \\
& r = \beta - E  y \label{sparsereg:const2}\\
& A  y \leq b \label{sparsereg:const3}\\
& C  y \leq D x \label{sparsereg:const4}\\
& \lambda \geq 0, \mu \geq 0 \label{sparsereg:const5}\\
& \lambda_{i} (a_{i}^{\top}  y - b_i) = 0, \quad \quad\forall i \in [m_1] \label{sparsereg:const6}\\
& \mu_{i}( c_{i}^{\top}  y - d_{i}^{\top} x) = 0, \quad \forall i \in [m_2] \label{sparsereg:const7}\\
& 2 \text{diag}\left(\left\{\frac{1}{\gamma x_i}\right\}_{i \in S} \right)  y_{S} + E_{S}^{\top} \alpha + A_{S}^{\top} \lambda + C_{S}^{\top} \mu + \tilde{g}_{S} = 0 \label{sparsereg:const8}\\
& 2 r + \alpha = 0 \label{sparsereg:const9}\\
& t_i = - \frac{1}{\gamma} \frac{ y_i^2}{x_i^2}  - D_{i}^{\top} \mu +c_i \quad\forall i \in S \label{sparsereg:const10}\\
& t_i \leq -\frac{\gamma}{4}(X_i^{\top} \alpha + A_{i}^{\top} \lambda + C_{i}^{\top} \mu + g_i)^2 - D_{i}^{\top} \mu + c_i. \quad \forall i \in S^{C}.\label{sparsereg:const11}
\end{align}
\end{subequations}
Conditions \eqref{sparsereg:const1}--\eqref{sparsereg:const9} are the KKT conditions for the following regression problem in a reduced space:
\begin{align}\label{sparsereg:QP}
\min_{y} \quad & \sum_{i \in S} \frac{1}{\gamma} \frac{y_i^2}{x_i} + \|r\|_{2}^{2} + \tilde{g}_{S}^{\top} y \\
\text{s.t.} \quad & r = \beta - E_{S} y \nonumber\\
& A_{S} y \leq b \nonumber\\
& C_{S} y \leq D x\nonumber,
\end{align}
and $ y_{S}, \alpha, \lambda,$ and  $\mu$ are the optimal primal and dual solutions. By comparing with Theorem ~1 in \cite{bertsimas2022scalable}, we can see that $\partial \bar f_{\text{persp}_{2}}$ attains the same set of cutting planes.

\section{Supplementary computational results} \label{app:suppl}

We give instance-wise comparisons for the instances generated by \citet{Frangioni2007}. 
\begin{landscape}
\scriptsize
\begin{table}[!th]\label{tab:appendix1}
    \caption{Instance-wise comparison results for portfolio selection problems when $n = 300$ .}
    \centering
    \begin{tabular}{cccccccccccccccccccc}
    \hline
     & & \multicolumn{3}{c}{MISOCP} & & \multicolumn{4}{c}{OA-BC}  & & \multicolumn{4}{c}{OA-persp} & & \multicolumn{4}{c}{OA-persp+ro} \\ 
   \cline{3-5} \cline{7-10} \cline{12-15} \cline{17-20} 
     Problem & $k$ & Gap & Time & Node & & Gap & Time & Node & Cut & &Gap & Time & Node & Cut & & Gap & Time & Node & Cut\\
pard300-0 & 6  & 0.16 & 600.04 & 654 & &  0.45 & 600.38 & 9427    & 1513& & 0.00 & 63.75  & 14554   & 1992 && 0.00 & 79.19  & 16502  & 1993  \\
pard300-0 & 8  & 0.30 & 600.07 & 635 & & 0.65 & 600.01 & 10587   & 1503 && 0.00 & 190.04 & 39244   & 3466 && 0.00 & 202.97 & 36332  & 3416   \\
pard300-0 & 10 & 0.14 & 600.02 & 564 && 0.45 & 600.61 & 20672   & 1358 && 0.00 & 233.07 & 54078   & 2314 && 0.00 & 387.61 & 48568  & 2263   \\
pard300-0 & nc & 0.48 & 600.18 & 498 && 0.64 & 600.70 & 508165  & 393  && 0.00 & 387.09 & 609246  & 475  && 0.00 & 471.19 & 634535 & 472  \\
\hline
pard300-1 & 6  & 0.00 & 358.01 & 551 && 0.14 & 600.34 & 7244    & 891  && 0.00 & 51.50  & 6644    & 915  && 0.00 & 31.37  & 6003   & 925    \\
pard300-1 & 8  & 0.06 & 600.07 & 628 && 0.36 & 600.34 & 14410   & 1493 && 0.00 & 80.85  & 21515   & 1878 && 0.00 & 90.29  & 21973  & 1829  \\
pard300-1 & 10 & 0.00 & 486.69 & 622 && 0.22 & 600.21 & 30051   & 1354 && 0.00 & 129.66 & 39900   & 1510& & 0.00 & 164.98 & 52168  & 1468 \\
pard300-1 & nc & 0.17 & 600.11 & 577 && 0.00 & 107.59 & 237566  & 116  && 0.00 & 67.04  & 252109  & 117  && 0.00 & 64.63  & 253601 & 117  \\
\hline
pard300-2 & 6  & 0.37 & 600.10 & 518 && 0.00 & 460.12 & 12199   & 1160 && 0.00 & 36.85  & 12538   & 1149 && 0.00 & 41.39  & 11514  & 1155  \\
pard300-2 & 8  & 0.20 & 600.19 & 475 && 0.00 & 585.66 & 29930   & 1400 && 0.00 & 76.23  & 29629   & 1380 && 0.00 & 80.21  & 29450  & 1398  \\
pard300-2 & 10 & 0.32 & 600.05 & 480 && 0.62 & 600.00 & 20871   & 1411 && 0.18 & 600.00 & 96071   & 3649 && 0.19 & 600.03 & 99064  & 3455  \\
pard300-2 & nc & 0.82 & 600.15 & 442 && 0.00 & 182.56 & 120650  & 288  && 0.00 & 59.91  & 96237   & 272  && 0.00 & 56.76  & 99555  & 285  \\
\hline
pard300-3 & 6  & 0.61 & 600.04 & 727 && 0.57 & 600.25 & 12023   & 1511 && 0.00 & 69.45  & 17349   & 1963 && 0.00 & 72.77  & 15152  & 1921 \\
pard300-3 & 8  & 0.46 & 600.07 & 704 && 0.76 & 600.01 & 12290   & 1495 && 0.00 & 190.47 & 38768   & 3198 && 0.00 & 199.00 & 38904  & 3178  \\
pard300-3 & 10 & 0.27 & 600.16 & 712 && 0.64 & 600.29 & 15660   & 1460 && 0.00 & 362.65 & 53598   & 3738 && 0.00 & 437.73 & 64013  & 3740  \\
pard300-3 & nc & 8.16 & 600.27 & 577 && 0.00 & 135.29 & 77614   & 275  && 0.00 & 35.35  & 77204   & 276  && 0.00 & 34.21  & 74277  & 282  \\
\hline
pard300-4 & 6  & 0.00 & 220.47 & 537 && 0.00 & 166.65 & 3011    & 423  && 0.00 & 11.48  & 3074    & 447  && 0.00 & 11.66  & 2827   & 429  \\
pard300-4 & 8  & 0.00 & 150.93 & 504 && 0.00 & 201.27 & 4923    & 504  && 0.00 & 14.78  & 4841    & 494  && 0.00 & 14.93  & 5097   & 490   \\
pard300-4 & 10 & 0.00 & 313.88 & 563 && 0.00 & 310.42 & 18959   & 712  && 0.00 & 36.47  & 12623   & 748  && 0.00 & 42.76  & 15618  & 759  \\
pard300-4 & nc & 0.64 & 600.15 & 593 && 0.00 & 50.10  & 46204   & 83   && 0.00 & 13.68  & 39924   & 78   && 0.00 & 11.22  & 33926  & 86  \\
\hline
pard300-5 & 6  & 0.00 & 134.41 & 521 && 0.00 & 173.29 & 2965    & 437  && 0.00 & 11.49  & 3050    & 428  && 0.00 & 12.16  & 2934   & 418   \\
pard300-5 & 8  & 0.00 & 265.28 & 574 && 0.00 & 301.53 & 6647    & 757  && 0.00 & 23.44  & 7598    & 745  && 0.00 & 24.93  & 7533   & 756    \\
pard300-5 & 10 & 0.39 & 600.06 & 718 && 0.77 & 600.26 & 19129   & 1437 && 0.69 & 600.01 & 63240   & 5211 && 0.66 & 600.01 & 69370  & 4972  \\
pard300-5 & nc & 0.97 & 600.04 & 606 && 0.00 & 183.41 & 211273  & 266  && 0.00 & 70.42  & 213307  & 270  && 0.00 & 87.83  & 205647 & 258   \\
\hline
pard300-6 & 6  & 0.13 & 600.04 & 469 && 0.00 & 133.93 & 1813    & 335  && 0.00 & 8.67   & 1721    & 330  && 0.00 & 10.68  & 1561   & 331   \\
pard300-6 & 8  & 0.26 & 600.10 & 459 && 0.00 & 129.34 & 2246    & 318  && 0.00 & 9.00   & 1904    & 289  && 0.00 & 9.88   & 2193   & 292  \\
pard300-6 & 10 & 0.00 & 307.04 & 365 && 0.00 & 91.63  & 2168    & 216  && 0.00 & 10.60  & 2750    & 261  && 0.00 & 9.80   & 3004   & 221  \\
pard300-6 & nc & 2.43 & 600.27 & 396 && 0.00 & 21.63  & 3455    & 41   && 0.00 & 3.65   & 3985    & 45   && 0.00 & 3.32   & 2400   & 42     \\
\hline
pard300-7 & 6  & 0.00 & 270.95 & 576 && 0.00 & 244.85 & 3842    & 620  && 0.00 & 16.59  & 3469    & 629  && 0.00 & 19.60  & 3899   & 613   \\
pard300-7 & 8  & 0.00 & 462.82 & 653 && 0.00 & 418.45 & 8075    & 1049 && 0.00 & 32.32  & 7947    & 1046 && 0.00 & 37.14  & 11121  & 1035  \\
pard300-7 & 10 & 0.00 & 317.12 & 579 && 0.00 & 430.79 & 16710   & 1004 && 0.00 & 71.12  & 25051   & 1041 && 0.00 & 65.78  & 20430  & 1006  \\
pard300-7 & nc & 1.26 & 600.14 & 613 && 0.00 & 243.62 & 115781  & 298  && 0.00 & 85.81  & 73289   & 308  && 0.00 & 117.56 & 98849  & 315   \\
\hline
pard300-8 & 6  & 0.00 & 500.10 & 568 && 0.00 & 404.61 & 7295    & 1025 && 0.00 & 31.58  & 7667    & 1046 && 0.00 & 38.90  & 7913   & 1032  \\
pard300-8 & 8  & 0.30 & 600.02 & 599 && 0.50 & 600.36 & 12524   & 1497 && 0.00 & 112.81 & 27362   & 2333 && 0.00 & 127.49 & 23651  & 2333  \\
pard300-8 & 10 & 0.00 & 568.35 & 572 && 0.44 & 600.31 & 24616   & 1364 && 0.00 & 442.80 & 85912   & 2473 && 0.00 & 534.75 & 98009  & 2552  \\
pard300-8 & nc & 1.66 & 600.13 & 486 && 0.00 & 194.54 & 164052  & 316  && 0.00 & 80.78  & 165395  & 315  && 0.00 & 68.98  & 144705 & 332  \\
\hline
pard300-9 & 6  & 0.00 & 251.57 & 490 && 0.00 & 219.89 & 3555    & 561  && 0.00 & 14.66  & 3240    & 560  && 0.00 & 18.45  & 3632   & 568   \\
pard300-9 & 8  & 0.00 & 228.12 & 482 && 0.00 & 318.92 & 5977    & 800  && 0.00 & 25.26  & 6042    & 835  && 0.00 & 26.72  & 5659   & 785   \\
pard300-9 & 10 & 0.00 & 271.77 & 538 && 0.00 & 203.33 & 4540    & 501  && 0.00 & 16.89  & 4721    & 496  && 0.00 & 17.75  & 3904   & 484    \\
pard300-9 & nc & 1.49 & 600.00 & 551 && 0.56 & 600.01 & 1025175 & 316  && 0.00 & 589.91 & 1251427 & 356  && 0.00 & 430.20 & 926608 & 344 

    \end{tabular}
\end{table}
\normalsize
\end{landscape}

\begin{landscape}
\scriptsize
\begin{table}[!th]\label{tab:appendix2}
    \caption{Instance-wise comparison results for portfolio selection problems when $n = 400$.}
    \centering
    \begin{tabular}{cccccccccccccccccccc}
    \hline
     & & \multicolumn{3}{c}{MISOCP} & & \multicolumn{4}{c}{OA-BC}  & & \multicolumn{4}{c}{OA-persp} & & \multicolumn{4}{c}{OA-persp+ro} \\ 
   \cline{3-5} \cline{7-10} \cline{12-15} \cline{17-20} 
     Problem & $k$ & Gap & Time & Node & & Gap & Time & Node & Cut & &Gap & Time & Node & Cut & & Gap & Time & Node & Cut\\
pard400-0 & 6  & 0.67 & 600.23 & 625 &  & 0.85 & 600.33 & 5162    & 1126 &  & 0.00 & 114.78 & 16102   & 2789 &  & 0.00 & 129.38 & 15852   & 2776 \\
pard400-0 & 8  & 0.53 & 600.13 & 691 &  & 0.69 & 600.35 & 7677    & 1109 &  & 0.00 & 390.52 & 47432   & 4873 &  & 000  & 411.78 & 51595   & 4781 \\
pard400-0 & 10 & 0.09 & 600.20 & 587 &  & 0.00 & 575.93 & 34205   & 930  &  & 0.00 & 68.89  & 20167   & 893  &  & 0.00 & 106.13 & 32837   & 943  \\
pard400-0 & nc & 2.79 & 600.15 & 601 &  & 1.23 & 600.01 & 820944  & 202  &  & 1.14 & 600.00 & 1066345 & 221  &  & 1.19 & 600.01 & 899022  & 250  \\
\hline
pard400-1 & 6  & 0.00 & 28.46  & 15  &  & 0.00 & 14.59  & 51      & 21   &  & 0.00 & 0.88   & 49      & 21   &  & 0.00 & 0.81   & 69      & 19   \\
pard400-1 & 8  & 0.00 & 67.75  & 31  &  & 0.00 & 15.71  & 26      & 20   &  & 0.00 & 0.86   & 89      & 20   &  & 0.00 & 1.08   & 167     & 19   \\
pard400-1 & 10 & 0.00 & 91.25  & 45  &  & 0.00 & 10.36  & 97      & 12   &  & 0.00 & 0.92   & 126     & 15   &  & 0.00 & 0.84   & 103     & 15   \\
pard400-1 & nc & 0.00 & 5.21   & 1   &  & 0.00 & 18.19  & 449     & 23   &  & 0.00 & 0.97   & 378     & 17   &  & 0.00 & 1.02   & 364     & 16   \\
\hline
pard400-2 & 6  & 1.09 & 600.01 & 330 &  & 0.00 & 87.69  & 1038    & 157  &  & 0.00 & 4.56   & 896     & 163  &  & 0.00 & 4.85   & 931     & 167  \\
pard400-2 & 8  & 0.58 & 600.19 & 327 &  & 0.00 & 161.98 & 2080    & 278  &  & 0.00 & 9.68   & 2268    & 263  &  & 0.00 & 9.84   & 1947    & 279  \\
pard400-2 & 10 & 0.14 & 600.24 & 332 &  & 0.00 & 208.95 & 10134   & 336  &  & 0.00 & 21.00  & 7510    & 314  &  & 0.00 & 19.44  & 7014    & 304  \\
pard400-2 & nc & 0.31 & 600.31 & 307 &  & 0.00 & 26.04  & 1616    & 32   &  & 0.00 & 2.97   & 1793    & 22   &  & 0.00 & 2.62   & 1071    & 24   \\
\hline
pard400-3 & 6  & 0.37 & 600.37 & 559 &  & 0.52 & 600.30 & 8111    & 1117 &  & 0.00 & 63.02  & 13294   & 1723 &  & 0.00 & 64.94  & 12359   & 1724 \\
pard400-3 & 8  & 0.33 & 600.32 & 521 &  & 0.57 & 600.02 & 7347    & 1087 &  & 0.00 & 251.99 & 38879   & 3672 &  & 0.00 & 250.50 & 36804   & 3631 \\
pard400-3 & 10 & 0.43 & 600.27 & 549 &  & 0.60 & 600.47 & 10242   & 1015 &  & 0.49 & 600.08 & 55783   & 3727 &  & 0.50 & 600.02 & 52198   & 3530 \\
pard400-3 & nc & 1.16 & 600.24 & 530 &  & 0.00 & 101.38 & 34868   & 115  &  & 0.00 & 22.26  & 22862   & 118  &  & 0.00 & 38.40  & 39031   & 121  \\
\hline
pard400-4 & 6  & 0.39 & 600.25 & 715 &  & 0.00 & 335.62 & 5146    & 627  &  & 0.00 & 21.17  & 4993    & 646  &  & 0.00 & 25.37  & 5081    & 648  \\
pard400-4 & 8  & 0.32 & 600.22 & 722 &  & 0.00 & 526.71 & 10625   & 951  &  & 0.00 & 40.43  & 11851   & 973  &  & 0.00 & 42.31  & 9458    & 965  \\
pard400-4 & 10 & 0.29 & 600.00 & 713 &  & 0.44 & 600.33 & 18136   & 1009 &  & 0.00 & 304.55 & 65101   & 2084 &  & 0.00 & 320.14 & 70781   & 2111 \\
pard400-4 & nc & 3.44 & 600.46 & 673 &  & 1.09 & 600.00 & 773944  & 233  &  & 1.09 & 600.03 & 836874  & 247  &  & 1.17 & 600.02 & 902718  & 228  \\
\hline
pard400-5 & 6  & 0.00 & 361.13 & 317 &  & 0.00 & 16.54  & 39      & 25   &  & 0.00 & 0.92   & 65      & 24   &  & 0.00 & 2.01   & 56      & 24   \\
pard400-5 & 8  & 0.00 & 213.83 & 295 &  & 0.00 & 12.74  & 55      & 14   &  & 0.00 & 0.77   & 31      & 14   &  & 0.00 & 1.42   & 34      & 14   \\
pard400-5 & 10 & 0.00 & 150.25 & 297 &  & 0.00 & 30.91  & 243     & 48   &  & 0.00 & 1.69   & 328     & 41   &  & 0.00 & 3.90   & 499     & 61   \\
pard400-5 & nc & 1.33 & 600.02 & 394 &  & 0.00 & 43.53  & 2240    & 56   &  & 0.00 & 7.34   & 5186    & 68   &  & 0.00 & 7.87   & 3325    & 66   \\
\hline
pard400-6 & 6  & 0.66 & 600.22 & 719 &  & 0.73 & 600.16 & 7635    & 1122 &  & 0.00 & 104.90 & 14965   & 2559 &  & 0.00 & 168.34 & 12845   & 2532 \\
pard400-6 & 8  & 0.50 & 600.19 & 716 &  & 0.65 & 600.56 & 6986    & 1113 &  & 0.00 & 219.46 & 34701   & 3380 &  & 0.00 & 296.05 & 31447   & 3350 \\
pard400-6 & 10 & 0.12 & 600.17 & 738 &  & 0.22 & 600.15 & 15369   & 1051 &  & 0.00 & 74.9   & 20279   & 1264 &  & 0.00 & 109.27 & 22383   & 1182 \\
pard400-6 & nc & 1.92 & 600.37 & 715 &  & 0.82 & 600.02 & 1304288 & 78   &  & 0.54 & 600.01 & 1421540 & 82   &  & 0.73 & 600.01 & 1566134 & 76   \\
\hline
pard400-7 & 6  & 1.07 & 600.25 & 488 &  & 1.00 & 600.13 & 6699    & 1098 &  & 0.00 & 147.22 & 21138   & 3189 &  & 0.00 & 157.47 & 19730   & 3234 \\
pard400-7 & 8  & 0.66 & 600.32 & 432 &  & 0.86 & 600.47 & 7015    & 1077 &  & 0.00 & 457.43 & 52436   & 5581 &  & 0.00 & 510.10 & 55535   & 5492 \\
pard400-7 & 10 & 0.36 & 600.00 & 460 &  & 0.49 & 600.59 & 20688   & 993  &  & 0.00 & 266.85 & 60777   & 2109 &  & 0.00 & 284.53 & 70818   & 2060 \\
pard400-7 & nc & 1.22 & 600.52 & 423 &  & 0.00 & 119.40 & 44938   & 105  &  & 0.00 & 64.92  & 98069   & 94   &  & 0.00 & 46.19  & 36017   & 101  \\
\hline
pard400-8 & 6  & 0.44 & 600.34 & 746 &  & 0.71 & 600.42 & 13874   & 1116 &  & 0.00 & 94.72  & 19385   & 2380 &  & 0.00 & 129.77 & 17848   & 2396 \\
pard400-8 & 8  & 0.48 & 600.03 & 741 &  & 0.62 & 600.03 & 9496    & 1103 &  & 0.00 & 386.84 & 62479   & 4231 &  & 0.00 & 428.76 & 61514   & 4246 \\
pard400-8 & 10 & 0.33 & 600.01 & 733 &  & 0.45 & 600.25 & 16468   & 1040 &  & 0.17 & 600.02 & 80269   & 3012 &  & 0.10 & 600.01 & 75529   & 2998 \\
pard400-8 & nc & 0.45 & 600.68 & 699 &  & 0.00 & 52.82  & 1867    & 74   &  & 0.00 & 4.54   & 1825    & 67   &  & 0.00 & 5.23   & 1744    & 67   \\
\hline
pard400-9 & 6  & 0.63 & 600.21 & 721 &  & 0.85 & 600.39 & 6717    & 1095 &  & 0.00 & 170.26 & 26773   & 3505 &  & 0.00 & 182.38 & 27287   & 3482 \\
pard400-9 & 8  & 0.51 & 600.00 & 740 &  & 0.68 & 600.29 & 7388    & 1051 &  & 0.00 & 495.54 & 53731   & 5548 &  & 0.00 & 524.17 & 53215   & 5515 \\
pard400-9 & 10 & 0.21 & 600.35 & 733 &  & 0.38 & 600.02 & 11406   & 996  &  & 0.11 & 600.03 & 72666   & 2755 &  & 0.17 & 600.02 & 80224   & 2675 \\
pard400-9 & nc & 2.31 & 600.34 & 694 &  & 0.99 & 600.02 & 1227572 & 128  &  & 0.92 & 600.01 & 1276260 & 127  &  & 0.95 & 600.02 & 1334464 & 129 

    \end{tabular}
\end{table}
\normalsize
\end{landscape}
\end{document}